\documentclass[reqno]{amsart}
\usepackage{amsmath}
\usepackage{amscd, amssymb, amsfonts, amsbsy}
\usepackage{enumerate}
\usepackage{hyperref}
\usepackage{verbatim}

\numberwithin{equation}{section}

\theoremstyle{plain}
\newtheorem{theorem}{Theorem}[section]
\newtheorem{lemma}[theorem]{Lemma}

\theoremstyle{definition}
\newtheorem{assumption}[theorem]{Assumption}

\theoremstyle{remark}

\makeatletter
\def\dashint{\operatorname%
{\,\,\text{\bf--}\kern-.98em\DOTSI\intop\ilimits@\!\!}}
\makeatother

\newcommand{\VMO}{\mathrm{VMO}}
\newcommand{\BMO}{\mathrm{BMO}}

\def\bC{\mathbb{C}}
\def\bH{\mathbb{H}}
\def\bR{\mathbb{R}}
\def\bN{\mathbb{N}}

\def\cB{\mathcal{B}}
\def\cC{\mathcal{C}}
\def\cE{\mathcal{E}}
\def\cH{\mathcal{H}}
\def\cK{\mathcal{K}}
\def\cL{\mathcal{L}}
\def\cM{\mathcal{M}}

\def\cX{\mathcal{X}}

\renewcommand{\vec}[1]{\boldsymbol{#1}}

\begin{document}
\title[$L_{p,q}$-estimates for elliptic and parabolic systems]{Weighted $L_{p,q}$-estimates for higher order elliptic and parabolic systems  with  $\mathrm{BMO}_x$ coefficients on Reifenberg flat domains}

\author[J. Choi]{Jongkeun Choi}
\address[J. Choi]{School of Mathematics, Korea Institute for Advanced Study, 85 Hoegiro, Dongdaemun-gu, Seoul 02455, Republic of Korea}
\email{jkchoi@kias.re.kr}

\author[D. Kim]{Doyoon Kim}
\address[D. Kim]{Department of Mathematics, Korea University, 145 Anam-ro, Seongbuk-gu, Seoul, 02841, Republic of Korea}
\email{doyoon\_kim@korea.ac.kr}
\thanks{D. Kim was supported by Basic Science Research Program through the National Research Foundation of Korea (NRF) funded by the Ministry of Education (2014R1A1A2054865).}


\begin{abstract}
We prove weighted $L_{p,q}$-estimates for divergence type higher order elliptic and parabolic systems with irregular coefficients on Reifenberg flat domains.
In particular, in the parabolic case the coefficients do not have any regularity assumptions in the time variable.
As functions of the spatial variables, the leading coefficients are permitted to have small mean oscillations.
The weights are in the class of Muckenhoupt weights $A_p$.
We also prove the solvability of the systems in weighted Sobolev spaces.
\end{abstract}

\maketitle

\section{Introduction}		\label{sec1}
We study weighted $L_{p,q}$-estimates and the solvability of divergence type higher order parabolic systems
\begin{equation}		\label{problem}
\vec u_t+(-1)^m\cL\vec u=\sum_{|\alpha|\le m} D^\alpha \vec f_\alpha 
\end{equation}
in $\Omega_T=(-\infty,T)\times \Omega$, where $T\in (-\infty,\infty]$ and $\Omega$ is either a bounded or unbounded domain in $\bR^d$.
The domain $\Omega$ can be $\bR^d$ as well.
The differential operator $\cL$ is in divergence form of order $2m$ acting on column vector valued functions $\vec u=(u^1,\ldots,u^n)^{\operatorname{tr}}$ defined on $\Omega_T$ as follows:
\[
\cL\vec u=\sum_{|\alpha|
\le m,\, |\beta| \le m}D^{\alpha}(A^{\alpha\beta}D^\beta\vec u).
\] 
Here, $\alpha=(\alpha_1,\ldots,\alpha_d)$, $\beta=(\beta_1,\ldots,\beta_d)$ are multi-indices,  and we write $D^\alpha \vec u=D_x^\alpha \vec u$ for  spatial derivatives of $\vec u$;
\[
D^\alpha\vec u=D^{\alpha_1}_1\ldots D^{\alpha_d}_d\vec u.
\]
All the coefficients $A^{\alpha\beta}=A^{\alpha\beta}(t,x)$ are  $n\times n$ complex-valued matrices whose entries $A^{\alpha\beta}_{ij}(t,x)$ are bounded measurable functions defined on the entire space $\bR^{d+1}$.
We also consider elliptic systems with the operator $\cL$ as in \eqref{problem}.
In this case all the coefficients and functions involved are independent of the time variable.

Throughout this paper,  the leading coefficients $A^{\alpha\beta}$, $|\alpha|=|\beta|=m$, satisfy the Legendre-Hadamard ellipticity condition, which is more general than the uniform ellipticity condition.
We assume that the coefficients of the parabolic systems are merely measurable in the time variable and have small bounded mean oscillations (BMO) with respect to the spatial variables.
We call such coefficients {\em $\mathit{BMO}_{\mathit{x}}$ coefficients} as in \cite{MR2352490} to indicate that the small mean oscillation condition is enforced only on the spatial variables.
As mentioned in \cite{MR1828321},  such type of coefficients with no regularity assumption in the time variable are necessary in the study of filtering theory.

In this paper,  we prove a priori weighted (mixed) norm estimates
 for divergence type higher order parabolic systems \eqref{problem} with $\BMO_x$ coefficients in  $\Omega_T$, where $\Omega$ is a Reifenberg flat domain (possibly unbounded).
More precisely, 
we first establish a priori weighted $L_p$-estimates with a Muckenhoupt weight $w\in A_p$:
$$
\sum_{|\alpha|\le m}\|D^\alpha \vec u\|_{L_{p,w}(\Omega_T)}\le N\sum_{|\alpha|\le m}\|\vec f_\alpha\|_{L_{p,w}(\Omega_T)}, \quad p\in (1,\infty).
$$
The weight $w$ is  defined on a space of homogeneous type $\cX$ in $\bR^{d+1}$ such that $\Omega_T \subset \cX$. 
For a precise definition of weights defined on spaces of homogeneous type, see Section \ref{161006@sec1}.
With regard to $L_p$-estimates with Muckenhoupt weights defined in the whole space, see \cite{MR1980981,MR2286441,MR3225808,MR3467697,MR3630407}.
We then obtain a priori weighted $L_{p,q}$-estimates with a mixed weight:
\begin{equation}		\label{161002@eq1}
\sum_{|\alpha|\le m}\|D^\alpha \vec u\|_{L_{p,q,w}(\Omega_T)}\le N\sum_{|\alpha|\le m}\|\vec f_\alpha\|_{L_{p,q,w}(\Omega_T)}, \quad p,\,q\in (1,\infty).
\end{equation}
The $L_{p,q}$-norm with the mixed weight $w=w_1w_2$ is defined as 
$$
\|f\|_{L_{p,q,w}(\Omega_T)}=\left(\int_{\cX_2}\left(\int_{\cX_1}|f|^pI_{\Omega_T}w_1 \right)^{q/p}w_2 \right)^{1/q},
$$
where  $w_1$ is an $A_p$ weight  on a space of homogeneous type $\cX_1\subset \bR^{d_1}$ and $w_2$ is an $A_q$ weight on a space of homogeneous type $\cX_2\subset \bR\times \bR^{d_2}$, $d_1+d_2=d$. 
Here, $\Omega_T\subset \cX_1\times \cX_2$ and  the mixed weight $w$ is defined on $\cX_1\times \cX_2 \subset \bR^{d+1}$.
For more discussions of mixed weights, see Section \ref{161006@sec1} in this paper and \cite{MR3812104}.
We also discuss the solvability for higher order parabolic systems \eqref{problem} with $\BMO_x$ coefficients in weighted Sobolev spaces.
The corresponding results for divergence type higher order elliptic systems with BMO coefficients are also addressed.

In the study of divergence type higher order parabolic systems,
to the best of the authors' knowledge, our results  are completely new  in the sense that Sobolev spaces with $A_p$ weights are considered, the coefficients are merely measurable in time, and the domains are Reifenberg flat.
Even in the case of  $L_p$-estimates (with unweighted and unmixed-norm) for higher order systems on Reifenberg flat domains, there are no such results in the literature. 
The only existing related literature seems to be a recent paper \cite{MR2771670}, where the authors proved $L_p$-estimates for higher order systems with $\BMO_x$ coefficients  in the whole space, on a half space, and on bounded Lipschitz domains.
Restricted to fourth-order parabolic systems with $\BMO_x$ coefficients, unweighted $L_p$-estimates on a bounded Reifenberg flat domain are obtained in \cite{MR2460025} under the uniform parabolicity condition.
With regard to the study of unweighted $L_p$-estimates for second order parabolic systems with $\BMO_x$ and VMO (vanishing mean oscillations) coefficients, see, for instance, \cite{MR2304157,MR2328932,MR2650802} and references therein.
Recently, in \cite[Sections 7 and 8]{MR3812104} the authors proved weighted $L_{p,q}$-estimates and the solvability for divergence type higher order parabolic systems \eqref{problem} when the leading coefficients are measurable in one spatial direction and have small BMO semi norms in the other variables including the time variable (partially BMO).
They also proved weighted $L_{p,q}$-estimates for non-divergence type higher order systems with $\BMO_x$ coefficients in the whole space and on half spaces.
See also \cite{MR2286441} for weighted $L_{p,q}$-estimates of non-divergence type higher order parabolic systems with time independent VMO coefficients in the whole space.

Main results in this paper as well as those in \cite[Sections 7 and 8]{MR3812104} can be viewed as a series of weighted estimates and the solvability for (higher order) elliptic and parabolic systems with irregular coefficients.
However, on establishing weighted $L_p$-estimates (unmixed norm), the arguments are  different.
In \cite{MR3812104}, the proof relies on  mean oscillation estimates combined with the  Fefferman-Stein sharp function theorem.
This argument was given by Krylov \cite{MR2304157} for the $L_p$ theory of  both divergence and non-divergence elliptic and parabolic equations with $\VMO_x$ coefficients.
In this paper, we adapt the idea in Caffarelli-Peral \cite{MR1486629}, which is based on a level set argument together with the measure theory on a ``crawling of ink spots" lemma originally due to Safonov and Krylov \cite{MR0563790}.
This argument is also used in \cite{MR2835999}  for  higher order systems with partially BMO coefficients in Reifenberg flat domains.
To establish weighted $L_{p}$-estimates for $p\in (1,\infty)$, we refine the measure theory  to  spaces of homogeneous type with weighted measures.
Moreover, we generalize the level set estimates from \cite{MR2835999}, roughly speaking, to control the weighted measure of  level sets of $|D^\alpha \vec u|$ by those of $\cM(|\vec f_\alpha|^q)^{1/q}$ with not only $q=2$ but also for any $q\in (1,\infty)$.
Here, $\cM$ is the Hardy-Littlewood maximal function operator.
This type of level set estimate with $q=2$ was used, for instance, in \cite{MR3225808,MR3467697} to obtain weighted $L_p$-estimates for divergence type second order systems.
A noteworthy difference is that in this paper, $L_p$-estimates are established with $A_p$ weights, whereas in \cite{MR3225808,MR3467697}  $L_p$-estimates are obtained with $A_{p/2}$ weights, the collection of which is strictly smaller than $A_p$.

Using aforementioned weighted $L_p$-estimates (unmixed norm), we prove weighted $L_{p,q}$-estimates  \eqref{161002@eq1} (mixed norm) for the systems.
The key ingredient  is a refined version of Rubio de Francia extrapolation theorem used in \cite{MR3812104}.
We remark that such a refinement of the well-known extrapolation theorem (see, for instance, \cite{MR2797562}) is necessary in our setting where the coefficients and the boundaries are very rough.
For more discussion, see the  paragraph above Theorem \ref{1017@@thm1} in this paper.
One can find in \cite{MR2286441,MR3630407} the well-known version of the extrapolation theorem employed to obtain $L_{p,q}$-estimates.
We also refer the reader to \cite{MR2352490,MR2764911,MR2982717,MR3611504} and references therein for more information about mixed-norm estimates with or without weights.

For the solvability of the systems in weighted Sobolev spaces, we adapt the idea in  \cite[Section 8]{MR3812104}, where a priori weighted estimates and the solvability in unweighted Sobolev spaces are used for the existence of solutions to the corresponding systems.
From the results on parabolic systems, we immediately obtain the corresponding results for higher order elliptic systems in divergence form with BMO coefficients. 

The remainder of this paper is organized as follows.
Section 2 contains  some notation and definitions. 
In Section 3, we state our main theorems including the results of higher order elliptic systems.
In Section 4, we present  some auxiliary results, while in Section 5, we establish interior and boundary estimates for derivatives of solutions.
Section 6 is devoted to the level set argument, and based on the results in Sections 4-6, we provide the proofs of the main theorems in Section 7.
In the appendix, we provide the proofs of some technical lemmas.

\section{Preliminaries}		\label{sec2}

\subsection{Basic notation}

We use $X=(t,x)$ to denote a point in $\bR^{d+1}$; $x=(x_1,\ldots,x_d)$ is a point in $\bR^d$.
We also write $Y=(s,y)$ and $X_0=(t_0,x_0)$, etc.
Let $\bR^d_+=\{x\in \bR^d:x_1>0\}$ and $\bR^{d+1}_+=\bR\times \bR^d_+$.
We  use $\Omega$ to denote an open set in $\bR^d$ and $\Omega_T=(-\infty,T)\times \Omega$.
We use the following notion for cylinders in $\bR^{d+1}$:
\begin{align*}
&Q_r(X)=(t-r^{2m},t)\times B_r(x), \\
&Q_r^+(X)=Q_r(X)\cap \bR^{d+1}_+,\\
&\cB_r(X)=(t-r^{2m},t+r^{2m})\times B_r(x),
\end{align*}
where $B_r(x)$ is the usual Euclidean ball of radius $r$ centered at $x\in \bR^d$.
Here, if we define the parabolic distance between the points $X$ and $Y$ in $\bR^{d+1}$ as 
\begin{equation}							\label{eq0215_01}
\rho(X,Y)=\max \big\{|x-y|, |t-s|^{\frac{1}{2m}}\big\},
\end{equation}
then 
$$
\cB_r(X)=\{Y\in \bR^{d+1}:\rho(X,Y)<r\},
$$
that is, $\cB_r(X)$ is an open ball in $\bR^{d+1}$ equipped with the parabolic distance $\rho$.
We abbreviate $Q_r=Q_r(0)$ and $B_r=B_r(0)$, etc.
For a function $f$ on $Q$, we use  $\left( f \right)_Q$ to denote the average of $f$ in $Q$, that is,
\[
(f)_Q=\frac{1}{|Q|}\int_Q f=\dashint_Q f.
\]

\subsection{Weights on a space of homogeneous type}		\label{161006@sec1}

Let $\cX$ be a set. 
A nonnegative symmetric function $\rho$ on $\cX\times \cX$ is called a quasi-metric on $\cX$ if there exists a positive constant $K_1$ such that
$$
\rho(x,x)=0\quad \text{and}\quad \rho(x,y)\le K_1\left(\rho(x,z)+\rho(z,y) \right)
$$
for any $x,y,z\in \cX$.
We denote balls in $\cX$ by
\begin{equation}
							\label{eq0624_01}
B_r^\cX(x)=\{y\in \cX:\rho(x,y)<r\}, \quad \forall x\in \cX, \quad \forall r>0.
\end{equation}
We say that $(\cX,\rho,\mu)$ is a space of homogeneous type if $\rho$ is a quasi-metric on $\cX$, $\mu$ is a Borel measure defined on a $\sigma$-algebra on $\cX$ which contains all the balls in $\cX$, and the following doubling property holds: there exists a constant $K_2$ such that for any $x\in \cX$ and $r>0$, 
$$
0<\mu(B_{2r}^\cX(x))\le K_2\mu(B_r^\cX(x))<\infty.
$$
Without loss of generality, we assume that balls $B_r^\cX(x)$ are open in $\cX$.

For any $p\in (1,\infty)$ and a space of homogeneous type $(\cX,\rho,\mu)$, the space $A_p(\cX)$ denotes the set of all nonnegative functions $w(x)$ on $\cX$ such that 
$$
[w]_{A_p}:=\sup_{\substack{x\in \cX \\ r>0}}\bigg(\dashint_{B_r^\cX(x)}w\,d\mu \bigg)\bigg(\dashint_{B_r^\cX(x)}w^{-\frac{1}{p-1}}\,d\mu\bigg)^{p-1}<\infty.
$$
One can easily check that  $[w]_{A_p} \ge 1$ for all $w \in A_p(\cX)$ and $[w]_{A_p}=1$ when $w\equiv1$.
We denote
\[
w(E)=\int_E w \, d\mu.
\]

Throughout this paper, whenever $\cX$ is said to be a space of homogeneous type in $\bR^k$ for some positive integer $k$, we mean the triple $(\cX, \rho, \mu)$, where $\cX$ is an open set in $\bR^k$, the metric $\rho$ is the usual Euclidean distance, and $\mu$ is the Lebesgue measure in $\bR^k$.
If $\cX$ is assumed to be a space of homogeneous type in $\bR \times \bR^k$ (or $\bR^{k+1}$), then $\cX$ is an open set in $\bR^{k+1}$, the metric $\rho$ is the parabolic distance defined in \eqref{eq0215_01}, and $\mu$ is the Lebesgue measure in $\bR^{k+1}$.
Thus, for example, when we consider weights of the type $w(t,x) = w_1(x') w_2(t,x'')$ in the mixed norm case for parabolic equations/systems,  where $w_1$ is a weight on a space of homogeneous type $\cX_1\subset \bR^{d_1}$ and $w_2$ is a weight on a space of homogeneous type $\cX_2\subset \bR \times \bR^{d_2}$, $d_1+ d_2 = d$,  $\cX_1$  is equipped with the usual Euclidean distance and the $d_1$-dimensional Lebesgue measure, and $\cX_2$  is equipped with  the parabolic distance $\rho$ and the $(d_2+1)$-dimensional Lebesgue measure.

Since we consider only the Lebesgue measures and the parabolic or Euclidean distances, the constant $K_1$ is always $1$ in our case. However, the doubling constant $K_2$ may vary depending on the choice of $\cX$.
For example, the doubling constants of the whole spaces $\bR^d$ and $\bR^{d+1}$ are $2^d$ and $2^{d+2m}$, respectively.
When $\cX\subset \bR^{d+1}$ and there exists a constant $\varepsilon>0$ such that $|\cB_r(X)\cap \cX |\ge \varepsilon |\cB_r(X)|$ for any $X\in \cX$ and $r>0$, $\cX$ is a space of homogeneous type with a doubling constant $K_2=K_2(d,m,\varepsilon)$.
If $\cX=\bR\times \Omega$, where $\Omega$ is a bounded Reifenberg flat domain in $\bR^d$, then the doubling constant of $\cX$ is determined by $d$, $m$, $|\Omega|$, $R_0$, and $\gamma\in (0,1/4)$, where $R_0$ and $\gamma$ are constants in Assumption \ref{0923.ass1}; see \cite[Remark 7.3]{MR3812104}.
Moreover, if $\cX$ is assumed to be $\cX_1\times \cX_2$, where $\cX_1$ and $\cX_2$ are spaces of homogeneous type with doubling constants $K_2'$ and $K_2''$ in $\bR^{d_1}$ and $\bR\times \bR^{d_2}$, $d_1+d_2=d$, respectively, 
then $\cX$ is a space of homogeneous type in $\bR^{d+1}$ with a doubling constant $K_2$, where $K_2$ is determined by $K_2'K_2''$; see Lemma \ref{160629@lem1}.

\subsection{Function spaces}

Let $p,\,q\in (1,\infty)$, $-\infty\le S<T\le \infty$, $\Omega$ be an open set in $\bR^d$, and $(S,T)\times \Omega\subseteq \cX_1\times \cX_2$, where 
 $\cX_1$ and $\cX_2$ are  spaces of homogeneous type in  $\bR^{d_1}$ and $\bR\times \bR^{d_2}$, $d_1+d_2=d$, respectively.
Let 
\[
w(t,x)=w(t,x',x'')=w_1(x')w_2(t,x''), \quad x'\in \cX_1, \quad (t,x'')\in \cX_2,
\]
where $w_1\in A_p(\cX_1)$ and $w_2\in A_q(\cX_2)$.
For such $w$, we define $L_{p,q,w}((S,T)\times \Omega)$ as the set of all measurable functions $u$ on $(S,T)\times \Omega$ having a finite norm
\begin{multline*}
\|u\|_{L_{p,q,w}((S,T)\times \Omega)}\\
=\left(\int_{\cX_2}\left(\int_{\cX_1} |u|^pI_{(S,T)\times \Omega} w_1(x')\,dx'\right)^{q/p}w_2(t,x'')\,dx''\,dt\right)^{1/q}.
\end{multline*}
We use 
\[
W^{1,m}_{p,q,w}((S,T)\times \Omega)=\{ u: u, Du, \ldots, D^m u, u_t\in L_{p,q,w}((S,T)\times \Omega)\}
\]
equipped with the norm
\[
\|u\|_{W^{1,m}_{p,q,w}((S,T)\times \Omega)}=\|u_t\|_{L_{p,q,w}((S,T)\times \Omega)}+\sum_{|\alpha| \le m}\|D^\alpha u\|_{L_{p,q,w}((S,T)\times \Omega)}.
\]
We also set 
\[
\bH^{-m}_{p,q,w}((S,T)\times \Omega)=\left\{f:f=\sum_{|\alpha| \le m}D^\alpha f_\alpha, \quad f_\alpha \in L_{p,q,w}((S,T)\times \Omega)\right\},
\]
\[
\|f\|_{\bH^{-m}_{p,q,w}((S,T)\times \Omega)}=\inf\left\{\sum_{|\alpha| \le m}\|f_\alpha\|_{L_{p,q,w}((S,T)\times \Omega)}: f=\sum_{|\alpha| \le m}D^\alpha f_\alpha \right\},
\]
and 
\begin{multline*}
\cH^m_{p,q,w}((S,T)\times \Omega)\\
=\left\{ u: u_t\in \bH^{-m}_{p,q,w}((S,T)\times \Omega), \,\, D^\alpha u\in L_{p,q,w}((S,T)\times \Omega), \,\, |\alpha| \le m\right\},
\end{multline*}
\[
\|u\|_{\cH^m_{p,q,w}((S,T)\times \Omega)}=\|u_t\|_{\bH^{-m}_{p,q,w}((S,T)\times \Omega)}+\sum_{|\alpha| \le m}\|D^\alpha u\|_{L_{p,q,w}((S,T)\times \Omega)}.
\]
We denote by $\mathring\cH^m_{p,q,w}((S,T)\times \Omega)$ the closure of $C^\infty_0([S,T]\times \Omega)$ in $\cH^m_{p,q,w}((S,T)\times \Omega)$, where $C^\infty_0([S,T]\times \Omega)$ is the set of all infinitely differentiable functions defined on $[S,T]\times \Omega$ with a compact support in $[S,T]\times \Omega$.
We abbreviate $L_{p,q,w}((S,T)\times \Omega)^n=L_{p,q,w}((S,T)\times \Omega)$, $L_{p,p,w}((S,T)\times \Omega)=L_{p,w}((S,T)\times \Omega)$, and $L_{p,q,1}((S,T)\times \Omega)=L_{p,q}((S,T)\times \Omega)$, etc.

For the elliptic case, we assume that $\Omega\subset \cX_1\times \cX_2$, where $\cX_1$ and $\cX_2$ are spaces of homogeneous type in $\bR^{d_1}$ and $\bR^{d_2}$, $d_1+d_2=d$, respectively.
Let 
\[
w(x)=w_1(x')w_2(x''), \quad x'\in \cX_1, \quad x''\in \cX_2,
\]
where $w_1\in A_p(\cX_1)$ and $w_2\in A_q( \cX_2)$.
For such $w$, we define $L_{p,q,w}( \Omega)$ as the set consisting of all measurable functions $u$ defined on $\Omega$ having a finite norm
$$
\|u\|_{L_{p,q,w}(\Omega)}=\left(\int_{ \cX_2}\left(\int_{\cX_1} |u|^pI_\Omega w_1(x')\,dx'\right)^{q/p}w_2(x'')\,dx''\right)^{1/q}.
$$
We also set 
$$
W^m_{p,q,w}(\Omega)=\left\{u: D^\alpha u\in L_{p,q,w}(\Omega), \, |\alpha|\le m\right\},
$$
$$
\|u\|_{W^m_{p,q,w}(\Omega)}=\sum_{|\alpha|\le m}\| D^\alpha u\|_{L_{p,q,w}(\Omega)}.
$$
We denote by $\mathring{W}^m_{p,q,w}(\Omega)$ the closure of $C_0^\infty(\Omega)$ in $W^m_{p,q,w}(\Omega)$.

\section{Main results}		\label{sec3}
Throughout this section, we assume that the coefficients of $\cL$ are bounded:
\begin{equation}		\label{boundedness}
\big|A^{\alpha\beta}_{ij}\big| \le 
\left\{
\begin{aligned}
\delta^{-1}&\quad \text{if }\, |\alpha|=|\beta|=m,\\
K &\quad \text{if }\, |\alpha|<m \,\text{ or } \, |\beta|<m,
\end{aligned}
\right.
\end{equation}
and the leading coefficients satisfy the  Legendre-Hadamard ellipticity condition:
\begin{equation}		\label{LH}
\Re \left(\sum_{|\alpha|=|\beta|=m} \sum_{i,j=1}^n A^{\alpha\beta}_{ij}(X)\xi^\alpha \xi^\beta \overline{\eta}_i\eta_j \right) \ge \delta|\xi|^{2m}|\eta|^2  
\end{equation}
for any $X\in \bR^{d+1}$, $\xi\in \bR^d$, and $\eta\in \bC^n$.
Here, we use the notation $\Re(f)$ to denote the real part of $f$.

To state our regularity assumption on the leading coefficients, we introduce the following notation. 
For a function $\vec g=( g^1,\ldots, g^n)^{\operatorname{tr}}$ on $\bR^{d+1}$, 
we define the mean oscillation of $\vec g$ in $Q_R(X_0)$ with respect to $x$ as
\begin{equation*}		
(\vec g)^{x,\sharp}_R(X_0)=\dashint_{Q_R(X_0)}\big|\vec g(s,y)-\dashint_{B_R(x_0)}\vec g(s,z) \, dz  \big| \,dy \, ds.
\end{equation*}

\begin{assumption}[$\gamma$]		\label{0923.ass1}
There exists $R_0\in (0,1]$ such that the following hold.
\begin{enumerate}[(i)]
\item
For any $X_0=(t_0,x_0)\in\bR\times \overline{\Omega}$ and $R\in(0,R_0]$ such that either $B_R(x_0)\subset \Omega$ or $x_0\in \partial \Omega$, we have
$$
\sum_{|\alpha|=|\beta|=m}(A^{\alpha\beta})^{x,\sharp}_R(X_0)\le \gamma.
$$
\item
For any $X_0=(t_0,x_0)\in \bR\times \partial \Omega$ and $R\in (0,R_0]$, 
there is a spatial coordinate system depending on $x_0$ and $R$ such that in this new coordinate system, we have
\begin{equation*}		
\{y:{x_0}_1+\gamma R<y_1\}\cap B_R(x_0)\subset \Omega_R(x_0)\subset \{y:{x_0}_1-\gamma R<y_1\}\cap B_R(x_0),
\end{equation*}
where ${x_0}_1$ is the first coordinate of $x_0$ in the new coordinate system.
\end{enumerate}
\end{assumption}

The main results of the paper read as follows.
We note that our results hold for both bounded and unbounded domains $\Omega\subseteq \bR^d$, and that if  $\Omega=\bR^d$, 
then Assumption \ref{0923.ass1} $(\gamma)$ is understood as Assumption \ref{0923.ass1} $(\gamma)$ $(i)$.

\begin{theorem}		\label{1008@thm1}
Let $T\in (-\infty,\infty]$, $\Omega$  be a domain in $\bR^d$, and $\Omega_T\subseteq \cX$, where  $\cX$ is  a space of homogeneous type in $\bR^{d+1}$ with a doubling constant $K_2$.
Let $p\in (1,\infty)$, $K_0\ge 1$, $w\in A_p(\cX)$, and $[w]_{A_p}\le K_0$.
Then there exist constants 
\begin{align*}
&\gamma=\gamma(d,m,n,\delta,p,K_0,K_2)\in (0,1/6),\\
&\lambda_0=\lambda_0(d,m,n,\delta,p, K_0,K_2, R_0,K)>0
\end{align*}
such that, under Assumption \ref{0923.ass1} $(\gamma)$, 
for $\vec u\in \mathring{\cH}^m_{p,w}(\Omega_T)$ satisfying 
\begin{equation}	\label{160608@eq5}
\vec u_t+(-1)^m\cL\vec u+\lambda\vec u=\sum_{|\alpha| \le m}D^\alpha \vec f_\alpha \quad \text{in }\, \Omega_T, 
\end{equation}
where $\lambda\ge \lambda_0$ and $\vec f_\alpha\in L_{p,w}(\Omega_T)$,  we have 
\begin{equation}		\label{1017@e3a}
\sum_{|\alpha| \le m}\lambda^{1-\frac{|\alpha|}{2m}}\|D^\alpha \vec u\|_{L_{p,w}(\Omega_T)}\le N\sum_{|\alpha| \le m}\lambda^{\frac{|\alpha|}{2m}}\|\vec f_\alpha\|_{L_{p,w}(\Omega_T)},
\end{equation}
where $N=N(d,m,n,\delta,p,K_0,K_2)$.
Moreover, for $\lambda \ge \lambda_0$ and $\vec f_\alpha\in L_{p,w}(\Omega_T)$, there exists a unique $\vec u\in \mathring{\cH}^m_{p,w}(\Omega_T)$ satisfying \eqref{160608@eq5}.
\end{theorem}

The next result is weighted $L_{p,q}$-estimates (mixed norms) for parabolic systems.

\begin{theorem}		\label{1016@thm1}
Let $T\in (-\infty,\infty]$, $\Omega$  be a domain in $\bR^d$, and $ \Omega_T\subseteq \cX_1\times \cX_2$, where $\cX_1$ and $\cX_2$ are  spaces of homogeneous type with  doubling constants $K_2'$ and $K_2''$ in $\bR^{d_1}$ and $\bR\times \bR^{d_2}$, $d_1+d_2=d$, respectively.
Let $p,q\in (1,\infty)$, $K_0\ge 1$, and 
$$
w(t,x)=w_1(x')w_2(t,x''), \quad x'\in \cX_1, \quad (t,x'')\in \cX_2,
$$
where $w_1\in A_p(\cX_1)$ with $[w_1]_{A_p}\le K_0$ and $w_2\in A_q(\cX_2)$ with $[w_2]_{A_q}\le K_0$.
Then there exist constants
\begin{align*}
&\gamma=\gamma(d,m,n,\delta,p,q,K_0, K_2', K_2'')\in (0,1/6),\\
&\lambda_0=\lambda_0(d,m,n,\delta,p,q, K_0,K_2',K_2'',R_0,K)>0,
\end{align*}
such that, under Assumption \ref{0923.ass1} $(\gamma)$, for $\vec u\in \mathring{\cH}^m_{p,q,w}(\Omega_T)$ satisfying 
\begin{equation}		\label{160621@eq1}
\vec u_t+(-1)^m\cL\vec u+\lambda\vec u=\sum_{|\alpha| \le m}D^\alpha \vec f_\alpha \quad \text{in }\, \Omega_T, 
\end{equation}
where $\lambda\ge \lambda_0$ and $\vec f_\alpha\in L_{p,q,w}(\Omega_T)$,  we have 
\begin{equation}		\label{1017@e4}
\sum_{|\alpha| \le m}\lambda^{1-\frac{|\alpha|}{2m}}\|D^\alpha \vec u\|_{L_{p,q,w}(\Omega_T)}\le N\sum_{|\alpha| \le m}\lambda^{\frac{|\alpha|}{2m}}\|\vec f_\alpha\|_{L_{p,q,w}(\Omega_T)},
\end{equation}
where $N=N(d,m,n,\delta,p,q,K_0, K_2', K_2'',d_1,d_2)$.
Moreover, for $\lambda \ge \lambda_0$ and $\vec f_\alpha\in L_{p,q,w}(\Omega_T)$, there exists a unique  $\vec u\in \mathring{\cH}^m_{p,q,w}(\Omega_T)$ satisfying \eqref{160621@eq1}.
\end{theorem}

If unmixed norms are considered, the elliptic case as in the theorem below is covered by  \cite{MR2835999}.
Here, we present  the mixed norm case for elliptic systems, which follows easily from Theorem \ref{1016@thm1} and a standard argument  in the proof of  \cite[Theorem 2.6]{MR2650802}.

\begin{theorem}		\label{160622@thm1}
Let $\Omega$ be a domain in $\bR^d$  and $ \Omega\subseteq \cX_1 \times \cX_2$, where $\cX_1$ and $ \cX_2$ are  spaces of homogeneous type with  doubling constants $K_2'$ and $K_2''$ in $\bR^{d_1}$ and $\bR^{d_2}$, $d_1+d_2=d$, respectively.
Let $p,q\in (1,\infty)$, $K_0\ge 1$, and 
$$
w(x)=w_1(x')w_2(x''), \quad x'\in \cX_1, \quad x''\in \cX_2,
$$
where $w_1\in A_p(\cX_1)$ with $[w_1]_{A_p}\le K_0$ and $w_2\in A_q( \cX_2)$ with $[w_2]_{A_q}\le K_0$.
Then there exist constants
\begin{align*}
&\gamma=\gamma(d,m,n,\delta,p,q,K_0, K_2', K_2'')\in (0,1/6),\\
&\lambda_0=\lambda_0(d,m,n,\delta,p,q,K_0, K_2',K_2'',R_0,K)>0,
\end{align*}
such that, under Assumption \ref{0923.ass1} $(\gamma)$, 
for $\vec u\in \mathring{W}^m_{p,q,w}(\Omega)$ satisfying 
\begin{equation}		\label{160622@@eq1}
(-1)^m\cL\vec u+\lambda\vec u=\sum_{|\alpha| \le m}D^\alpha \vec f_\alpha \quad \text{in }\, \Omega, 
\end{equation}
where $\lambda\ge \lambda_0$ and $\vec f_\alpha\in L_{p,q,w}(\Omega)$, we have 
$$
\sum_{|\alpha| \le m}\lambda^{1-\frac{|\alpha|}{2m}}\|D^\alpha \vec u\|_{L_{p,q,w}(\Omega)}\le N\sum_{|\alpha| \le m}\lambda^{\frac{|\alpha|}{2m}}\|\vec f_\alpha\|_{L_{p,q,w}(\Omega)},
$$
where $N=N(d,m,n,\delta,p,q,K_0, K_2', K_2'',d_1,d_2)$.
Moreover, for $\lambda \ge \lambda_0$ and $\vec f_\alpha\in L_{p,q,w}(\Omega)$, there exists a unique
$\vec u\in \mathring{W}^m_{p,q,w}(\Omega)$ satisfying \eqref{160622@@eq1}.
\end{theorem}

\section{Some auxiliary results}			\label{sec4}

The results in this section can be found, for instance, in \cite{MR1232192,MR3243734}.
Here, we present those results in a form convenient for later use along with some of their proofs.
In particular, we specify the parameters on which the constants $N$ and $\mu$ in the results below depend.
For example, we assume that $[w]_{A_p} \le K_0$, $K_0 \ge 1$, and show that the constants $N$ in the inequalities depend on $K_0$ rather than $[w]_{A_q}$.

In this section, we assume that $\cX$ is a space of homogeneous type in $\bR^{d+1}$ (resp. $\bR^d$) with the distance $\rho$ in \eqref{eq0215_01} (resp. the usual Euclidean distance), the Lebesgue measure, and a doubling constant $K_2$.
In the case that $\cX \subset \bR^{d+1}$, as we recall, $\cB_r^{\cX}(X)$ is a ball in $\cX$ defined by 
$$
\cB_r^{\cX}(X)=\{Y\in \cX:\rho(X,Y)<r\}.
$$

\begin{lemma}[Reverse H\"older's inequality]		\label{1016@lem1}
Let $p\in (1,\infty)$, $K_0 \ge 1$, $w\in A_p(\cX)$, and  $[w]_{A_p}\le K_0$.
Then there exist constants $\mu_0>1$ and $N>0$, depending only on $p$, $K_0$, and $K_2$, such that 
$$
\bigg(\dashint_{\cB_r^{\cX}(X)} w^{\mu_0}\,dY\bigg)^{\frac{1}{\mu_0}}\le N\dashint_{\cB_r^{\cX}(X)}w\,dY
$$
for any $X\in \cX$ and $r>0$.
\end{lemma}

\begin{proof}
See \cite[Theorem 7.3.3]{MR3243734} or \cite[Theorem 3, p. 212]{MR1232192}.
\end{proof}

\begin{lemma}		\label{1016@lem2}
Let $p\in (1,\infty)$, $K_0 \ge 1$, $w\in A_p(\cX)$, and  $[w]_{A_p}\le K_0$.
Then there exist constants $p_0\in (1,p)$ and $N>0$, depending only on $p$, $K_0$, and $K_2$, such that $w\in A_{p_0}(\cX)$ and 
\begin{equation}		\label{160614@eq2}
\bigg(\dashint_{\cB_r^{\cX}(X)}w^{-\frac{1}{p_0-1}}\,dY\bigg)^{p_0-1}\le N \bigg(\dashint_{\cB_r^{\cX}(X)}w^{-\frac{1}{p-1}}\,dY\bigg)^{p-1}
\end{equation}
for any $X\in \cX$ and $r>0$.
\end{lemma}

\begin{proof}
Let $p'=\frac{p}{p-1}$.
Since $w\in A_p(\cX)$, we obtain that   $v=w^{-\frac{1}{p-1}}\in A_{p'}(\cX)$.
Indeed, for any $X\in \cX$ and $r>0$, we have 
$$
\bigg(\dashint_{\cB^{\cX}_r(X)}v\,dY\bigg)\bigg(\dashint_{\cB^{\cX}_r(X)}v^{-\frac{1}{p'-1}}\,dY\bigg)^{p'-1}
$$
$$
=\bigg(\dashint_{\cB^{\cX}_r(X)}w^{-\frac{1}{p-1}}\,dY\bigg)\bigg(\dashint_{\cB^{\cX}_r(X)}w\,dY\bigg)^{\frac{1}{p-1}}
\le [w]_{A_p}^{\frac{1}{p-1}}\le K_0^{\frac{1}{p-1}}.
$$
Therefore, by Lemma \ref{1016@lem1}, we have 
\[
\bigg(\dashint_{\cB^{\cX}_r(X)}v^{\mu_0}\,dY\bigg)^{\frac{1}{\mu_0}}\le N\dashint_{\cB^{\cX}_r(X)}v\,dY,
\]
where $(N,\mu_0)=(N,\mu_0)(p, K_0, K_2)$.
By taking $p_0=\frac{p-1}{\mu_0}+1\in (1,p)$ in the above inequality, we obtain \eqref{160614@eq2}.
\end{proof}

\begin{lemma}		\label{1015@lem3}
Let $p\in (1,\infty)$, $K_0 \ge 1$, $w\in A_p(\cX)$, and  $[w]_{A_p}\le K_0$. 
Then there exist  constants $\mu_1\in (0,1)$ and  $N_1>0$, depending only on $p$, $K_0$, and $K_2$, such that for any measurable set $E\subset \cX$, we have 
\[
\frac{1}{K_0}\left(\frac{|E\cap \cB_r^{\cX}(X)|}{|\cB_r^{\cX}(X)|}\right)^p\le \frac{w(E\cap \cB_r^{\cX}(X))}{w(\cB_r^{\cX}(X))}\le N_1\left(\frac{|E\cap \cB_r^{\cX}(X)|}{|\cB_r^{\cX}(X)|}\right)^{\mu_1}
\]
for any $X\in \cX$ and $r>0$.
\end{lemma}

\begin{proof}
From H\"older's inequality and the definition of $A_p$, it follows that 
\begin{align*}
|E\cap \cB_r^{\cX}(X)|&=\int_{E\cap \cB_r^{\cX}(X)}w^{\frac{1}{p}}w^{-\frac{1}{p}}\,dY\\
&\le w(E\cap \cB_r^{\cX}(X))^{\frac{1}{p}}\bigg(\dashint_{\cB_r^{\cX}(X)}w^{-\frac{1}{p-1}}\,dY\bigg)^{\frac{p-1}{p}}|\cB_r^{\cX}(X)|^{\frac{p-1}{p}}\\
& \le [w]_{A_p}^{\frac{1}{p}} w(E\cap \cB_r^{\cX}(X))^{\frac{1}{p}}\bigg(\dashint_{\cB_r^{\cX}(X)}w\,dY\bigg)^{-\frac{1}{p}}|\cB_r^{\cX}(X)|^{\frac{p-1}{p}}\\
&=[w]_{A_p}^{\frac{1}{p}}\left(\frac{w(E\cap \cB_r^{\cX}(X))}{w(\cB_r^{\cX}(X))}\right)^{\frac{1}{p}}|\cB_r^{\cX}(X)|,
\end{align*}
which gives  the first inequality.
For the second inequality, we observe that H\"older's inequality and Lemma \ref{1016@lem1} imply that
\begin{align*}
w(E\cap \cB_r^{\cX}(X))&\le |E\cap \cB_r^{\cX}(X)|^{\frac{\mu_0-1}{\mu_0}}|\cB_r^{\cX}(X)|^{\frac{1}{\mu_0}}\bigg(\dashint_{\cB_r^{\cX}(X)} w^{\mu_0}\, dY\bigg)^{\frac{1}{\mu_0}}\\
&\le N|E\cap \cB_r^{\cX}(X)|^{\frac{\mu_0-1}{\mu_0}}|\cB_r^{\cX}(X)|^{\frac{1}{\mu_0}}\bigg(\dashint_{\cB_r^{\cX}(X)} w\, dY\bigg)\\
&=N\left(\frac{|E\cap \cB_r^{\cX}(X)|}{|\cB_r^{\cX}(X)|}\right)^{\frac{\mu_0-1}{\mu_0}}w(\cB_r^{\cX}(X)),
\end{align*}
where $(N,\mu_0)=(N,\mu_0)(p,K_0,K_2)$.
The lemma is proved.
\end{proof}

The following Hardy-Littlewood maximal function theorem with $A_p$ weights was obtained in \cite{MR0740173}.
Below, we denote the maximal function of $f$ defined on $\cX$ by 
\begin{equation}		\label{160617@eq1}
\cM f(X)=\sup_{\substack{Z\in \cX, r>0\\ X\in \cB_r^{\cX}(Z)}}\dashint_{\cB_r^{\cX}(Z)}|f(Y)|\, dY.
\end{equation}

\begin{theorem}		\label{1008@thm5}
Let $p\in (1,\infty)$, $K_0 \ge 1$, $w\in A_p(\cX)$, and $[w]_{A_p}\le K_0$.
Then for any $f\in L_{p,w}(\cX)$, we have 
\[
\|\cM f\|_{L_{p,w}(\cX)}\le N\|f\|_{L_{p,w}(\cX)},
\]
where $N=N(p,K_0, K_2)>0$.
\end{theorem}

\section{Interior and boundary estimates}		\label{sec5}

In this section, we denote
\begin{equation*}
\cL_0\vec u=\sum_{|\alpha|=|\beta|=m}D^\alpha\big({A}^{\alpha\beta}_0D^\beta\vec u\big),
\end{equation*}
where $A_0^{\alpha\beta}=A_0^{\alpha\beta}(t)$ satisfy \eqref{boundedness} and \eqref{LH}.

In the lemma below, we provide $L_\infty$-estimates not only for a weak solution $\vec u$ but also for its derivatives $D^m \vec u$.
In fact, the results in the lemma is proved by the standard iteration argument along with the Sobolev embedding theorem and the known $L_p$-estimates for  systems.
Precisely, since the coefficients $A_0^{\alpha\beta}$ are independent of the spatial variables, we view the operator $\cL_0$ as a non-divergence type operator and use the $L_p$-estimates for non-divergence type systems proved in \cite{MR2771670}.
The proof is mostly standard, so we only describe the major steps.

For a given constant $\lambda \ge 0$ and functions $\vec u$ and $\vec f_\alpha$, $|\alpha| \le m$, we write
\begin{equation}
							\label{eq0223_01}		
U=\big(\lambda^{\frac{1}{2}-\frac{|\alpha|}{2m}}D^\alpha \vec u\big)_{|\alpha| \le m} \quad \text{and}\quad F=\big(\lambda^{\frac{|\alpha|}{2m}-\frac{1}{2}}\vec f_\alpha\big)_{|\alpha| \le m},
\end{equation}
where $f_\alpha \equiv 0$ for $|\alpha|<m$ whenever $\lambda = 0$.

\begin{lemma}		\label{0929.lem1}
Let $\lambda\ge 0$ and $q\in (1,\infty)$.
\begin{enumerate}[$(a)$]
\item
If 
$\vec u\in C^\infty_{\operatorname{loc}}((-\infty,0] \times \bR^d)$
satisfies
\begin{equation*}				
\vec u_t+(-1)^m\cL_0\vec u+\lambda\vec u=0 \quad \text{in }\, Q_2,
\end{equation*}
then we have 
\begin{equation}		\label{1002.eq1}
\|U\|_{L_\infty(Q_1)}\le N\|U\|_{L_q(Q_2)},
\end{equation}
where $N=N(d,m,n,\delta,q)$.
\item
If 
$\vec u\in C^\infty_{\operatorname{loc}}((-\infty,0] \times \overline{\bR^d_+})$ satisfies
\begin{equation*}		
\left\{
\begin{aligned}
\vec u_t+(-1)^m\cL_0\vec u+\lambda \vec u=0 &\quad \text{in }\, Q_2^+,\\
|\vec u| =\cdots=|D^{m-1}_1\vec u|=0 &\quad \text{on }\, Q_2 \cap \{X \in \bR^{d+1}: x_1 = 0\},
\end{aligned}
\right.
\end{equation*}
then we have
$$
\|U\|_{L_\infty(Q_1^+)}\le N\|U\|_{L_q(Q_2^+)},
$$
where $N=N(d,m,n,\delta,q)$.

\end{enumerate}
\end{lemma}

\begin{proof}
We first prove the assertion (a) with $\lambda=0$.
As mentioned above, owing to the coefficients being independent of the spatial variables, $\vec u \in C_{\operatorname{loc}}^\infty((-\infty,0] \times \bR^d)$ satisfies the following non-divergence type system
\begin{equation}
							\label{eq0222_04}
\vec u_t+(-1)^m \sum_{|\alpha|=|\beta|=m} A_0^{\alpha\beta} D^\alpha D^\beta \vec u = 0 \quad \text{in }\, Q_2.
\end{equation}
By the $L_p$-estimate for non-divergence type systems in \cite[Theorem 2]{MR2771670} and the localization argument as in the proof of \cite[Lemma 1]{MR2771670}, we obtain
$$
\|\vec u\|_{W_p^{1,2m}(Q_r)}\le N\|\vec u\|_{L_p(Q_R)}
$$
for any $p \in (1,\infty)$ and $1 \le r < R \le 2$, where $N=N(d,m,n,\delta,p,r,R)$.
From the above inequality, the standard iteration argument, and Sobolev embedding type results (see Lemmas \ref{0922.lem2} and \ref{0922.lem4}), we have 
\begin{equation}				\label{eq0222_03}
\|\vec u\|_{L_\infty(Q_1)} \le N\|\vec u\|_{W^{1,2m}_{q_1}(Q_1)}\le N\|\vec u\|_{L_q(Q_2)},
\end{equation} 
where $q_1 \in [q,\infty)$ is sufficiently large so that $2m > (d+2m)/q_1$ (see Lemma \ref{0922.lem2}). 
Since $D^m \vec u$ also satisfies \eqref{eq0222_04}, we obtain \eqref{eq0222_03} with $D^m \vec u$ in place of $\vec u$, which is  \eqref{1002.eq1} with $\lambda = 0$.

Now we prove the assertion (b) for $\lambda = 0$. We repeat the above argument by using $L_p$-estimates for systems defined on a half space with the Dirichlet boundary condition. 
Precisely, using \cite[Theorem 4]{MR2771670} instead of \cite[Theorem 2]{MR2771670}, we arrive at
\begin{equation}				\label{eq0222_05}
\|D^m \vec u\|_{L_\infty(Q_1^+)} \le N\|\vec u\|_{W^{1,2m}_{q_1}(Q_1^+)}\le N\|\vec u\|_{L_q(Q_2^+)},
\end{equation}
where $q_1 \in [q,\infty)$ and $m > (d+2m)/q_1$ (again, see Lemma \ref{0922.lem2}). 
Then, one can bound the last term in \eqref{eq0222_05} by $N \|D^m\vec u\|_{L_q(Q_2^+)}$ by repeatedly using the Poincar\'e inequality (see, for instance, \cite[Theorem 10.2.5]{MR2435520}), i.e.,
for $u \in W_p^1(B_1^+)$, we have
$$
\int_{B_1^+}|u(x)|^p \, dx=\int_{B_1}|\bar{u}(x)|^p \, dx
=\frac{1}{2|B_1|}\int_{B_1}\int_{B_1 \setminus B_1^+}|\bar{u}(x)-\bar{u}(y)|^p \, dy \, dx
$$
$$
\le \frac{1}{2|B_1|}\int_{B_1}\int_{B_1}|\bar{u}(x)-\bar{u}(y)|^p \, dy \, dx
\le 2^d \int_{B_1}|\nabla \bar{u}(x)|^p \, dx,
$$
where $\bar{u} \in W_p^1(B_1)$ is an extension of $u$ to $B_1$ so that $u \equiv 0$ on $B_1 \setminus B_1^+$.

For a general $\lambda>0$, we only prove (a). The other case is entirely analogous. We use an idea by S. Agmon.
Let $\eta=\eta_\lambda(\tau)$ be a smooth function on $\bR$ defined by 
\[
\eta(\tau)=\cos(\lambda^{\frac{1}{2m}}\tau)+\sin (\lambda^{\frac{1}{2m}}\tau ).
\]
Note that 
\[
(-1)^m D^{2m}_\tau\eta=\lambda \eta \quad \text{and}\quad |D^j_\tau \eta(0)|=\lambda^{\frac{j}{2m}}, \quad \forall j=0,1,\ldots.
\]
By setting 
\[
\hat{\vec u}(t,x,\tau)=\vec u(t,x)\eta(\tau) \quad \text{and}\quad \widehat{Q}_r=(-r^{2m},0)\times \{(x,\tau)\in \bR^{d+1}:|(x,\tau)|<r\},
\]
we see that $\hat{\vec u}$ satisfies 
\[
\hat{\vec u}_t+(-1)^m \cL_0 \hat{\vec u}+(-1)^mD_\tau^{2m}\hat{\vec u}=0 \quad \text{in }\, \widehat{Q}_2.
\]
By applying the  result for $\lambda = 0$ to $\hat{\vec u}$, we obtain 
\begin{equation}		\label{0922.eq1d}
\|D^m_x \vec u\|_{L_\infty(Q_1)}\le\|D^m_{(x,\tau)} \hat{\vec u}\|_{L_\infty(\widehat{Q}_1)}\le N\|D^m_{(x,\tau)}\hat{\vec u}\|_{L_q(\widehat{Q}_2)}.
\end{equation}
We also obtain from \eqref{eq0222_03} that 
\begin{equation}		\label{0922.eq1e}
\|\vec u\|_{L_\infty(Q_1)}\le \|\hat{\vec u}\|_{L_\infty(\widehat{Q}_1)}\le N\|\hat{\vec u}\|_{L_q(\widehat{Q}_2)}\le N\|\vec u\|_{L_q(Q_2)}.
\end{equation}
Notice that $D^m_{(x,\tau)}\hat{\vec u}$ is a linear combination of 
\[
\lambda^{\frac{1}{2}-\frac{k}{2m}}\cos (\lambda^{\frac{1}{2m}}\tau)D^k_x\vec u \quad \text{and}\quad \lambda^{\frac{1}{2}-\frac{k}{2m}}\sin (\lambda^{\frac{1}{2m}}\tau)D^k_x\vec u, \quad k=0,1,\ldots,m.
\]
Therefore, by combining \eqref{0922.eq1d} and \eqref{0922.eq1e}, and then, using the interpolation inequalities, we conclude \eqref{1002.eq1}.
\end{proof}

In the following lemma,  we consider the operator $\cL$ without lower order terms, i.e., 
\[
\cL \vec u=\sum_{|\alpha|=|\beta|=m}D^\alpha(A^{\alpha\beta}D^\beta \vec u).
\]

\begin{lemma}		\label{0929-lem1}
Let $T\in [0,\infty]$,  $\lambda\ge 0$, $q\in (1,\infty)$, $\nu \in (1,\infty)$, $\nu' = \nu/(\nu-1)$, and $\Omega$ be a domain in $\bR^d$.
Assume $\vec u\in C^\infty_0((-\infty,T]\times \Omega)$ satisfies
\begin{equation}		\label{1004.eq1}
\vec u_t+(-1)^m\cL\vec u+\lambda \vec u=\sum_{|\alpha| \le m}D^\alpha \vec f_\alpha \quad \text{in }\ \Omega_T, 
\end{equation}
where $\vec f_\alpha\in L_{q,\operatorname{loc}}((-\infty,T] \times \overline{\Omega})$, $|\alpha| \le m$.
\begin{enumerate}[$(a)$]
\item
Suppose that Assumption \ref{0923.ass1} $(\gamma)$ {\rm (i)} holds at $0\in \Omega$ with $\gamma>0$. 
Then for $R$ such that $0 < R \le \min(R_0,\operatorname{dist} (0,\partial \Omega))$,  $\vec u$ admits a decomposition 
\[
\vec u=\vec v+\vec w \quad \text{in }\, Q_R
\]
satisfying
\begin{align}
\label{0923.eq2}
(|W|^q)^{\frac{1}{q}}_{Q_R}&\le N\gamma^{\frac{1}{q\nu'}}(|U|^{q\nu} )^{\frac{1}{q\nu}}_{Q_R}+N(|F|^q)^{\frac{1}{q}}_{Q_R},\\
\label{0923.eq2a}
\|V\|_{L_\infty(Q_{R/4})}&\le N\gamma^{\frac{1}{q\nu'}}(|U|^{q\nu})^{\frac{1}{q\nu}}_{Q_R}+N(|F|^q)^{\frac{1}{q}}_{Q_R}+N(|U|^q)^{\frac{1}{q}}_{Q_R}.
\end{align}

\item
Suppose that Assumption \ref{0923.ass1} $(\gamma)$ holds at $0\in \partial \Omega$ with $\gamma\in \big(0,\frac{1}{6}\big)$.
Then for $R\in (0,R_0]$,  $\vec u$ admits a decomposition 
\[
\vec u=\vec v+\vec w \quad \text{in }\, \cC_R:=Q_R\cap \Omega_T
\]
satisfying
\begin{align}
\label{0923.eq2b}
(|W|^q)^{\frac{1}{q}}_{\cC_R}&\le N\gamma^{\frac{1}{q\nu'}}(|U|^{q\nu} )^{\frac{1}{q\nu}}_{\cC_R}+N(|F|^q)^{\frac{1}{q}}_{\cC_R},\\
\label{0923.eq2c}
\|V\|_{L_\infty(\cC_{R/6})}&\le N\gamma^{\frac{1}{q\nu'}}(|U|^{q\nu})^{\frac{1}{q\nu}}_{\cC_R}+N(|F|^q)^{\frac{1}{q}}_{\cC_R}+N(U^q)^{\frac{1}{q}}_{\cC_R}.
\end{align}
\end{enumerate}
Here, the constant $N$ depends on $d$, $m$, $n$, $\delta$, $\nu$, and $q$, and
$V$ and $W$ are defined in the same way as $U$ in \eqref{eq0223_01} with $\vec u$ replaced by $\vec v$ and $\vec w$, respectively.
\end{lemma}

\begin{proof}
The proof is an adaptation of that of \cite[Lemma 8.3]{MR2835999}.
We may assume that $A^{\alpha\beta}$ and $\vec f_\alpha$ are infinitely differentiable.
If not, we take the standard mollifications and prove the estimates for the mollifications.
Then we can pass to the limit because the constants $N$ in the estimates are independent of the regularity of $A^{\alpha\beta}$ and $\vec f_\alpha$.
We further assume $\lambda>0$.
Otherwise, we add the term $\varepsilon \vec u$, $\varepsilon>0$, to both sides of \eqref{1004.eq1} and obtain the estimates for the modified system.
Then we let $\varepsilon \to 0^+$.

To prove the assertion $(a)$, we define 
\[
\cL_0\vec u=\sum_{|\alpha|=|\beta|=m}D^\alpha(A^{\alpha\beta}_0D^\beta \vec u), 
\]
where 
\[
A^{\alpha\beta}_0(t)=\dashint_{B_R}A^{\alpha\beta}(t,y)\,dy.
\]
Let $\varphi$ be a smooth function on $\bR^{d+1}$ satisfying
$$
0\le \varphi\le 1, \quad \operatorname{supp} \varphi\subset \cB_R, \quad \text{and}\quad \varphi\equiv 1\text{ on } Q_{R/2}.
$$
By \cite[Theorem 1]{MR2771670}, there exists a unique $\vec w\in \cH^m_q(\bR^d_{0})$ satisfying
\begin{multline*}
\vec w_t+(-1)^m\cL_0 \vec w+\lambda \vec w\\
=(-1)^m\sum_{|\alpha|=|\beta|=m}D^\alpha\big(\varphi(A^{\alpha\beta}_0-A^{\alpha\beta})D^\beta \vec u\big)+\sum_{|\alpha| \le m}D^\alpha(\varphi \vec f_\alpha)
\end{multline*}
in $ \bR^d_0$, where as we recall $\bR^d_0 = (-\infty,0) \times \bR^d$, and 
$$		
\|W\|_{L_q(\bR^d_0)}\le N\sum_{|\alpha|=|\beta|=m}\|(A^{\alpha\beta}_0-A^{\alpha\beta})D^\beta\vec u\|_{L_q(Q_R)}
$$
$$
+N\sum_{|\alpha| \le m}\lambda^{\frac{|\alpha|}{2m}-\frac{1}{2}}\|\vec f_\alpha\|_{L_q(Q_R)},
$$
where $N=N(d,m,n,\delta,q)$.
This together with H\"older's inequality gives \eqref{0923.eq2}.
Since all functions and coefficients involved are infinitely differentiable, by the classical parabolic theory, $\vec w$ is infinitely differentiable.
Therefore, the function   $\vec v=\vec u-\vec w$ is also infinitely differentiable, and it satisfies 
\[
\vec v_t+(-1)^m\cL_0 \vec v+\lambda \vec v=0 \quad \text{in }\, Q_{R/2}.
\]
By Lemma \ref{0929.lem1} $(a)$ with scaling,  we obtain  
\begin{equation*}	
\|V\|_{L_\infty(Q_{R/4})}\le N(|V|^q)^{\frac{1}{q}}_{Q_R}\le N(|U|^q)^{\frac{1}{q}}_{Q_R}+N(|W|^q)^{\frac{1}{q}}_{Q_R}.
\end{equation*}
Thus, we obtain \eqref{0923.eq2a} by using the above inequality and \eqref{0923.eq2}.

Next, we prove the assertion $(b)$.
Without loss of generality, we may assume that Assumption \ref{0923.ass1} $(\gamma)$ holds at $0$ in the original $(t,x)$-coordinates.
Define $\cL_0$ and $\varphi$ as above.
Consider a smooth function $\chi=\chi_R$ defined on $\bR$ such that 
$$
\chi(x_1)\equiv 0 \text{ for } x_1\le \gamma R, \quad \chi(x_1)\equiv 1 \text{ for } x_1\ge 2\gamma R,
$$
$$
|D^k\chi| \le N(\gamma R)^{-k} \text{ for } k=1,\ldots, m.
$$
Then, $\hat{\vec u}(X)=\chi(x_1)\vec u(X)$ along with all its derivatives vanishes on $Q_R\cap \{x_1\le \gamma R\}$ and satisfies in $Q_R^{\gamma+}:=Q_R\cap \{ x_1>\gamma R\}$,
$$
\hat{\vec u}_t+(-1)^m\cL_0\hat{\vec u}+\lambda\hat{\vec u}=(-1)^m \sum_{|\alpha|=|\beta|=m}D^\alpha\big(\big(A_0^{\alpha\beta}-A^{\alpha\beta})D^\beta\vec u\big)
$$
$$
+\sum_{|\alpha| \le m}\chi D^\alpha\vec f_\alpha+(-1)^m\vec g+(-1)^m \vec h,
$$
where we set
\[
\vec g=\cL_0((\chi-1)\vec u) \quad \text{and}\quad \vec h=(1-\chi)\cL\vec u.
\]
Let $\hat{\vec w}$ be the unique $\mathring{\cH}^m_q((-\infty,T) \times \{x:x_1>\gamma R\})$  solution of the problem  (see \cite[Theorem 3]{MR2771670}):
\begin{multline*}		
\hat{\vec w}_t+(-1)^m \cL_0\hat{\vec w}+\lambda \hat{\vec w}=(-1)^m\sum_{|\alpha|=|\beta|=m}D^\alpha\big(\varphi(A_0^{\alpha\beta}-A^{\alpha\beta})D^\beta\vec u\big)\\
+\sum_{|\alpha| \le m}\chi D^\alpha(\varphi \vec f_\alpha)+(-1)^m\hat{\vec g}+(-1)^m\hat{\vec h}
\end{multline*}
in $(-\infty,T)\times \{x:x_1>\gamma R\}$, where 
\begin{align*}
&\hat{\vec g}=\sum_{|\alpha|=|\beta|=m}D^\alpha\big(A_0^{\alpha\beta}\varphi D^\beta((\chi-1)\vec u)\big),\\
&\hat{\vec h}=(1-\chi)\sum_{|\alpha|=|\beta|=m}D^\alpha(A^{\alpha\beta}\varphi D^\beta \vec u).
\end{align*}
By using the argument as in \cite[Lemma A.1]{MR3812104}, we obtain 
\begin{equation}		\label{1007@eq1}
\sum_{k=0}^m \lambda^{\frac{1}{2}-\frac{k}{2m}}\big(I_{Q^{\gamma+}_R}|D^k\hat{\vec w}|^q\big)^{\frac{1}{q}}_{\cC_R}\le N\gamma^{\frac{1}{q\nu'}}\left(|U|^{q\nu}\right)^{\frac{1}{q\nu}}_{\cC_R}+N(|F|^q)^{\frac{1}{q}}_{\cC_R}.
\end{equation}
We extend $\hat{\vec w}$ to be zero in $\cC_R \setminus Q^{\gamma+}_R$, so that $\hat{\vec w}\in \cH^m_q(\cC_R)$.
Set
\[
\vec w=\hat{\vec w}+(1-\chi)\vec u \quad \text{and}\quad \vec v=\vec u-\vec w.
\]
Then similar to (7.19) in \cite{MR2835999}, we deduce \eqref{0923.eq2b} from \eqref{1007@eq1}.
Moreover, we find that
$\vec v\equiv 0$ in $\cC_R\setminus Q^{\gamma+}_R$ and $\vec v$ satisfies
$$
\left\{
\begin{aligned}
\vec v_t+(-1)^m\cL_0\vec v+\lambda\vec v=0 &\quad \text{in }\, Q_{R/2}\cap \{x_1>\gamma R\},\\
|\vec v|=\cdots=|D^{m-1}_1\vec v|=0 &\quad \text{on }\, Q_{R/2}\cap \{x_1=\gamma R\}.
\end{aligned}
\right.
$$
We write $X_0=(0,x_0)\in \bR^{d+1}$, where $x_0=(\gamma R, 0,\ldots,0)\in \bR^d$. 
Then we have 
\begin{align*}
(Q_{R/6}\cap \{x_1>\gamma R\}) &\subset (Q_{R/6}(X_0)\cap \{x_1>\gamma R\})\\
&\subset (Q_{R/3}(X_0)\cap \{x_1>\gamma R\})\subset (Q_{R/2}\cap \{x_1>\gamma R\}).
\end{align*}
Therefore, by applying Lemma \ref{0929.lem1} $(b)$ with scaling, we obtain 
\begin{align*}
\|V\|_{L_\infty(\cC_{R/6})}&=\|V\|_{L_\infty(Q_{R/6}\cap \{x_1>\gamma R\})}\\
&\le \|V\|_{L_\infty(Q_{R/6}(X_0)\cap \{x_1>\gamma R\})}\\
&\le N(|V|^q)^{1/q}_{Q_{R/3}(X_0)\cap \{x_1>\gamma R\}}\le N(|V|^q)^{1/q}_{\cC_{R/2}},
\end{align*}
which together with \eqref{0923.eq2b} gives \eqref{0923.eq2c}.
\end{proof}

\section{Level set argument}		\label{sec6}

In this section, we consider the operator $\cL$ without lower order terms, i.e.,
\[
\cL \vec u=\sum_{|\alpha|=|\beta|=m}D^\alpha(A^{\alpha\beta}D^\beta \vec u).
\]
We denote 
\begin{equation}		\label{160617@eq2}
\cC_r(X)=Q_r(X)\cap \Omega_T.
\end{equation}
If $\cX$ is a space of homogeneous type in $\bR^{d+1}$, then since $\cX$ is open in $\bR^{d+1}$ and we use the parabolic distance, we see that
$$
\cB_r^\cX(X)=\cB_r(X)\cap \cX,
$$
where, as we recall, $\cB_r^\cX(X)$ is a ball in $\cX$ defined in \eqref{eq0624_01}.

For a function $f$ on $\cX$, we define its maximal function $\cM f$ by \eqref{160617@eq1}.
We also denote for $s>0$,  $\nu\in (1,\infty)$, $\nu'=\nu/(\nu-1)$, and $q\in (1,\infty)$ that 
\begin{equation}		\label{1008@eq1}
\begin{aligned}
\cE_1(s)&=\{X\in \Omega_T: |U|(X)>s\},\\
\cE_2(s)&=\big\{X\in  \Omega_T:\gamma^{-\frac{1}{q\nu'}}(\cM(I_{\Omega_T}|F|^{q})(X))^{\frac{1}{q}}+(\cM(I_{\Omega_T}|U|^{q\nu})(X))^{\frac{1}{q\nu}}>s\big\},
\end{aligned}
\end{equation}
where $U$ and $F$ are as in \eqref{eq0223_01}.

\begin{lemma}		\label{1015@lem1}
Let $T \in (-\infty,\infty]$, $\nu \in (1,\infty)$, $\nu' = \nu/(\nu-1)$, $\Omega$ be a domain in $\bR^d$,  and $\Omega_T\subseteq \cX$, where $\cX$ is  a space of homogeneous type in $\bR^{d+1}$ with a doubling constant $K_2$.
Let 
$p_0,q\in (1,\infty)$, $\hat{K}_0\ge 1$, $w\in A_{p_0}(\cX)$, and $[w]_{A_{p_0}}\le \hat{K}_0$.
Suppose that  Assumption \ref{0923.ass1} $(\gamma)$ holds with $\gamma\in \big(0,\frac{1}{6}\big)$, and $\vec u\in C^\infty_0((-\infty,T]\times \Omega)$ satisfies
\begin{equation*}		
\vec u_t+(-1)^m\cL\vec u+\lambda\vec u=\sum_{|\alpha| \le m}D^\alpha \vec f_\alpha \quad \text{in }\,  \Omega_T,
\end{equation*}
where $\lambda>0$ and $\vec f_\alpha\in L_{q,\operatorname{loc}}((-\infty,T] \times \overline{\Omega})$, $|\alpha|\le m$.
Then there exists a constant $\kappa=\kappa(d,m,n,\delta,\nu,q)>1$ such that the following holds: for $X\in (-\infty,T]\times\overline{\Omega}$, $R\in (0,R_0]$, and $s>0$,
if
\begin{equation*}		
N_1\gamma^{\frac{\mu_1}{\nu'}}\le\frac{w\big(\cB_{R/64}^\cX(X)\cap \cE_1(\kappa s)\big)}{w\big(\cB_{R/64}^\cX(X)\big)},
\end{equation*}
where $(N_1,\mu_1)=(N_1,\mu_1)(p_0,\hat{K}_0,K_2)$ are from Lemma \ref{1015@lem3}, then we have 
\begin{equation}		\label{1015@eq1a}
\cC_{R/64}(X) \subset \cE_2(s).
\end{equation}
\end{lemma}

\begin{proof}
Let  $X=(t,x)\in (-\infty,T]\times \overline{\Omega}$, $R\in (0,R_0]$, and $s>0$.
Owing to Lemma \ref{1015@lem3}, we have 
$$
\frac{w\big(\cB_{R/64}^\cX(X)\cap \cE_1(\kappa s)\big)}{w\big(\cB_{R/64}^\cX(X)\big)}\le N_1\left(\frac{|\cB_{R/64}^\cX(X) \cap \cE_1(\kappa s)|}{|\cB_{R/64}^\cX(X)|}\right)^{\mu_1}.
$$
Therefore, it suffices to claim that \eqref{1015@eq1a} holds, provided that 
\begin{equation}		\label{1015@eq1}		
\gamma^{\frac{1}{\nu'}}<\frac{|\cB_{R/64}^\cX(X)\cap \cE_1(\kappa s)|}{|\cB_{R/64}^\cX(X)|}.
\end{equation}
By dividing $\vec u$ and $\vec f_\alpha$ by $s$, we may assume $s=1$.
We prove the claim by contradiction.
Suppose that 
\[
\gamma^{-\frac{1}{q\nu'}}(\cM(I_{\Omega_T}|F|^{q})(Z))^{\frac{1}{q}}+(\cM(I_{\Omega_T}|U|^{q\nu})(Z))^{\frac{1}{q\nu}}\le 1
\]
for some $Z\in \cC_{R/64}(X)$.
Set 
$$
T^*=\min \big(t+(R/64)^{2m},T\big) \quad \text{and}\quad X^*=(T^*,x)\in (-\infty,T]\times \overline{\Omega}.
$$
If $\operatorname{dist} (x,\partial \Omega)\ge R/8$,  we have 
\[
Z\in \cC_{R/64}(X)\subset Q_{R/32}(X^*)\subset Q_{R/8}(X^*)\subset\Omega_T.
\]
By Lemma \ref{0929-lem1} (a), $\vec u$ admits a decomposition $\vec u=\vec v+\vec w$ in $Q_{R/8}(X^*)$ with the estimates
\[
(|W|^{q})_{Q_{R/8}(X^*)}\le N_2\gamma^{\frac{1}{\nu'}} \quad \text{and}\quad \|V\|_{L_\infty(Q_{R/32}(X^*))}\le N_2,
\]
where $N_2=N_2(d,m,n,\delta,\nu,q)$.
From this together with Chebyshev's inequality, it follows that 
\begin{align}
\nonumber
|\cB_{R/64}^{\cX}(X)\cap \cE_1(\kappa)|&\le \{Y\in Q_{R/32}(X^*):|U|(Y)>\kappa\}\\
\nonumber
&\le \{Y\in Q_{R/32}(X^*):|W|(Y)>\kappa-N_2\}\\
\nonumber
&\le \int_{Q_{R/32}(X^*)}\left|\frac{W}{\kappa-N_2}\right|^{q} \ dY\le \frac{N_2\gamma^{\frac{1}{\nu'}}|Q_{R/8}(X^*)|}{|\kappa-N_2|^{q}}.\\
\label{160804@eq3}
&\le \frac{N_2'\gamma^{\frac{1}{\nu'}}|\cB_{R/64}^\cX (X)|}{|\kappa-N_2|^{q}},
\end{align}
where $N_2'=N_2'(d,m,n,\delta,\nu,q)$
and the last inequality is due to 
$$
|Q_{R/8}(X^*)|\le N(d)|Q_{R/64}(X)|\le N(d)|\cB^\cX_{R/64}(X)|.
$$
The estimate \eqref{160804@eq3} contradicts with \eqref{1015@eq1} if we choose a sufficiently large $\kappa$.

On the other hand, if $\operatorname{dist} (x, \partial \Omega)<R/8$,  we take $x_0\in \partial \Omega$ such that $\operatorname{dist} (x,\partial \Omega)=|x-x_0|$.
Note that
\[
Z\in \cC_{R/64}(X)\subset \cC_{R/6}(X_0^*)\subset \cC_R(X_0^*), \quad X_0^*=(T^*,x_0).
\]
By Lemma \ref{0929-lem1} (b), $\vec u$ admits a decomposition $\vec u=\vec v+\vec w$ in $\cC_R(X_0^*)$ with the estimates
\[
(|W|^{q})_{\cC_R(X_0^*)}\le N_3\gamma^{\frac{1}{\nu'}} \quad \text{and}\quad \|V\|_{L_\infty(\cC_{R/6}(X_0^*)}\le N_3,
\]
where $N_3=N_3(d,m,n,\delta,\nu,q)$.
Therefore, we obtain 
\begin{align}
\nonumber
|\cB_{R/64}^\cX(X)\cap \cE_1(\kappa)|&\le\{Y\in \cC_{R/6}(X^*_0):|U|(Y)>\kappa\}\\
\nonumber
&\le \{Y\in \cC_{R/6}(X^*_0):|W|(Y)>\kappa-N_3\}\\
\label{160804@eq6}
&\le \int_{\cC_{R/6}(X_0^*)}\left|\frac{W}{\kappa-N_3}\right|^q\ dY\le \frac{N_3\gamma^{\frac{1}{\nu'}}|\cC_R(X_0^*)|}{|\kappa-N_3|^q}.
\end{align}
Using the fact that 
\begin{equation}		\label{160804@eq1}
|\Omega\cap B_r(y)|\ge N(d)r^d, \quad \forall y\in \overline{\Omega}, \quad \forall r\in (0,R_0],
\end{equation}
we have 
$$
|\cC_R(X^*_0)|\le N(d)R^{d+2m} \le N(d)|\cB_{R/64}^{\cX}(X)|.
$$
Thus, from \eqref{160804@eq6}, we obtain that 
$$
|\cB_{R/64}^\cX(X)\cap \cE_1(\kappa)|\le \frac{N_3'\gamma^{\frac{1}{\nu'}}|\cB_{R/64}^{\cX}(X)|}{|\kappa-N_3|^q},
$$
where $N_3'=N_3'(d,m,n,\delta,\nu,q)$, which contradicts with \eqref{1015@eq1} if we choose a sufficiently large $\kappa$.
Thus, the claim is proved.
\end{proof}

\begin{lemma}		\label{1016@lem5}
Let $T\in (-\infty,\infty]$, $\Omega$  be a domain  in $\bR^d$, and $\Omega_T\subseteq \cX$, where 
$\cX$ is  a space of homogeneous type in $\bR^{d+1}$ with a doubling constant $K_2$.
Let $p\in (1,\infty)$, $K_0\ge 1$, $w\in A_p(\cX)$, and   $[w]_{A_p}\le K_0$.
Then there exists a constant 
\[
\gamma=\gamma(d,m,n,\delta,p, K_0,K_2)\in (0,1/6)
\]
such that, under  Assumption \ref{0923.ass1} $(\gamma)$, the following holds:
if $\vec u\in C^\infty_0((-\infty,T]\times \Omega)$ vanishes outside $Q_{\gamma R_0}(X_0)$, where $X_0\in \bR^{d+1}$, and satisfies 
\[
\vec u+(-1)^m \cL\vec u+\lambda\vec u=\sum_{|\alpha| \le m}D^\alpha \vec f_\alpha \quad \text{in }\, \Omega_T,
\]
where $\lambda>0$ and $\vec f_\alpha\in L_{p,w}( \Omega_T)$, then we have 
\begin{equation}		\label{1013@eq2}
\|U\|_{L_{p,w}( \Omega_T)}\le N\|F\|_{L_{p,w}(\Omega_T)},
\end{equation}
where  $N=N(d,m,n,\delta,p,K_0,K_2)$.
\end{lemma}

\begin{proof}
By Lemma \ref{1016@lem2}, we see that 
$$
w\in A_{p_0}(\cX) \quad \text{and}\quad [w]_{A_{p_0}}\le \hat{K}_0
$$
for some constants $p_0\in (1,p)$ and $\hat{K}_0\ge1$, depending only on $p$, $K_0$, and $K_2$.
We denote
$$
q=\frac{1}{2}\bigg(1+\frac{p}{p_0}\bigg), \quad \nu=\frac{p}{p_0q}, \quad \nu'=\frac{\nu}{\nu-1},
$$
and  let $(N_1,\mu_1)=(N_1,\mu_1)(p_0,\hat{K}_0,K_2)$ and  $\kappa=\kappa(d,m,n,\delta,\nu,q)$ be constants  in Lemma \ref{1015@lem3} and Lemma \ref{1015@lem1}, respectively.
We recall the notation \eqref{160617@eq2} and \eqref{1008@eq1}, and we remark that $\vec f_\alpha\in L_{q,\operatorname{loc}}((-\infty,T]\times \overline{\Omega})$.
Indeed, by using H\"older's inequality and $p_0q<p$, for any $X\in \cX$ and $R>0$ we have 
$$
\begin{aligned}
&\int_{\cB_R^{\cX}(X)} |\vec f_\alpha|^qI_{\Omega_T}\,dX\\
&\le \bigg(\int_{\cB_R^{\cX}(X)}|\vec f_\alpha|^{p_0q}I_{\Omega_T}w\,dX\bigg)^{\frac{1}{p_0}}\bigg(\int_{\cB_R^{\cX}(X)}w^{-\frac{1}{p_0-1}}\,dX\bigg)^{\frac{p_0-1}{p_0}}\\
&\le N\bigg(\int_{\cB_R^{\cX}(X)}|\vec f_\alpha|^{p}I_{\Omega_T}w\,dX\bigg)^{\frac{q}{p}}<\infty.
\end{aligned}
$$
Let  $\gamma\in (0,1/6)$ be a constant satisfying 
$$
N_1\gamma^{\frac{\mu_1}{\nu'}}<1.
$$
Since $\operatorname{supp}\vec u\subset Q_{\gamma R_0}(X_0)$, 
it suffices to prove the lemma when 
$$
\operatorname{supp}\vec u\subset \cB_{2\gamma R_0}(X_0), \quad  X_0\in (-\infty,T]\times \overline{\Omega}.
$$
We first claim that for any $X\in \Omega_T$ and $R\ge R_0/64$, we have 
\begin{equation}		\label{1012@e2}
w(\cE_1(\kappa s)\cap \cB_{R}^{\cX}(X))<N_1 \gamma^{\frac{\mu_1}{\nu'}}w(\cB^{\cX}_{R}(X)),
\end{equation}
provided that 
$$
s>s_0:= \frac{N_2}{N_1\gamma^{\mu_1/\nu'}\kappa w(\cB_{R_0/3}^\cX(X_0))^{1/p_0}}\|U\|_{L_{p_0,w}(\Omega_T)}, \quad N_2=N_2(K_2).
$$
Since $\operatorname{supp} \vec u\subset \cB_{2\gamma R_0}(X_0)$, we only need to consider the case when $\cB_{R}^\cX(X)\cap \cB_{2\gamma R_0}(X_0)\neq \emptyset$.
In this case, we have
$$
\cB_{R_0/3}^{\cX}(X_0)\subset \cB_{45R}^{\cX}(X), 
$$
and thus, by H\"older's inequality and the doubling property of $\cX$, we obtain
\begin{align*}
w(\cE_1(\kappa s)\cap \cB_{R}^{\cX}(X))&\le \frac{1}{\kappa s}\int_{\cB_{45R}^{\cX}(X)}I_{\Omega_T}|U|w\,dY\\
&\le \frac{1}{\kappa s}w(\cB_{45R}^{\cX}(X))^{1-\frac{1}{p_0}}\|U\|_{L_{p_0,w}(\Omega_T)}\\
&\le \frac{N_2}{\kappa s}\frac{w(\cB_{R}^{\cX}(X))}{w(\cB^{\cX}_{R_0/3}(X_0))^{1/p_0}}\|U\|_{L_{p_0,w}(\Omega_T)},
\end{align*}
where $N_2=N_2(K_2)$, which implies \eqref{1012@e2}.
Therefore, by using \eqref{1012@e2}, Lemma \ref{1015@lem1}, and ``the crawling of ink spots" lemma due to Safonov-Krylov \cite{MR0563790},
we have the following inequality;
\begin{equation}		\label{160804@eq9}
w(\cE_1(\kappa s))\le N\gamma^{\frac{\mu_1}{\nu'}}w(\cE_2(s)), \quad \forall s>s_0,
\end{equation}
where $N=N(d,m,p,K_0,K_2)$.
We provide a detailed proof of \eqref{160804@eq9} in Lemma \ref{lem0221_1} in Appendix  for the reader's convenience (also see \cite{MR3467697}).

By \eqref{160804@eq9}, we obtain
\begin{align}
\nonumber
\|U\|_{L_{p,w}(\Omega_T)}^p&=p\int_0^\infty w(\cE_1(s))s^{p-1}\,ds =p\kappa^p\int_0^\infty w(\cE_1(\kappa s))s^{p-1}\,ds\\
\nonumber
&\le N\int_0^{s_0} w(\cE_1(\kappa s))s^{p-1}\,ds+N\gamma^{\frac{\mu_1}{\nu'}}\int_{0}^\infty w(\cE_2( s))s^{p-1}\,ds\\
\label{1013@eq1b}
&:=I_1+I_2,
\end{align}
where $N=N(d,m,n,\delta,\nu,p,K_0, K_2)$.
Notice from Chebyshev's inequality that
$$
w(\cE_1(\kappa s))\le (\kappa s)^{-p_0}\|U\|_{L_{p_0,w}(\Omega_T)}^{p_0}, \quad \forall s>0.
$$
Using this together with H\"older's inequality and Lemma \ref{1015@lem3},  we have 
\begin{align*}
I_1&\le N \left(\int_0^{s_0}s^{p-p_0-1}\,ds\right)\|U\|^{p_0}_{L_{p_0,w}(\Omega_T)}\\
&\le N\gamma^{\frac{\mu_1}{\nu'}(p_0-p)}\left(\frac{w(\cB^\cX_{2\gamma R_0}(X_0))}{w(\cB^{\cX}_{R_0/3}(X_0))}\right)^{\frac{p-p_0}{p_0}}\|U\|_{L_{p,w}(\Omega_T)}^p\\
&\le N\gamma^{\frac{\mu_1}{\nu'}(p_0-p)}\left(\frac{|\cB^\cX_{2\gamma R_0}(X_0)|}{|\cB^{\cX}_{R_0/3}(X_0)|}\right)^{\mu_1\frac{p-p_0}{p_0}}\|U\|_{L_{p,w}(\Omega_T)}^p.
\end{align*}
Therefore, since (use \eqref{160804@eq1})
$$
\frac{|\cB^\cX_{2\gamma R_0}(X_0)|}{|\cB^{\cX}_{R_0/3}(X_0)|}\le \frac{|\cB_{2\gamma R_0}(X_0)|}{|\cC_{R_0/3}(X_0)|}\le N\gamma^{d+2m},
$$
we have
$$
I_1\le N\gamma^{\mu_1(p-p_0)\left(\frac{d+2m}{p_0}-\frac{1}{\nu'}\right)}\|U\|^p_{L_{p,w}(\Omega_T)}.
$$
To estimate $I_2$, we note that  $p/q>p/(q\nu)=p_0$ and  
$$
[w]_{A_{p/q}}\le [w]_{A_{p/(q\nu)}} \le \hat{K}_0.
$$
Then by the definition of $\cE_2$ and Theorem \ref{1008@thm5} with $p/q$ in place of $p$, we obtain   
\begin{align}		
I_2
\nonumber
&\le N\gamma^{\frac{\mu_1}{\nu'}}\left(\gamma^{-\frac{p}{q\nu'}}\big\|(\cM(I_{\Omega_T}|F|^q))^{\frac{1}{q}}\big\|^p_{L_{p,w}(\cX)}+\big\|(\cM(I_{\Omega_T}|U|^{q\nu})^{\frac{1}{q\nu}}\big\|^p_{L_{p,w}(\cX)}\right)\\
\label{1016@e1a}
&\le N\gamma^{\frac{1}{\nu'}\left(\mu_1-\frac{p}{q}\right)}\|F\|_{L_{p,w}(\Omega_T)}^p+N\gamma^{\frac{\mu_1}{\nu'}}\|U\|_{L_{p,w}(\Omega_T)}^p.
\end{align}
Finally, by combining \eqref{1013@eq1b}--\eqref{1016@e1a},  and then, choosing a sufficiently small $\gamma$, we conclude \eqref{1013@eq2}.
\end{proof}

\section{Proofs of main theorems}		\label{sec7}


We begin with the proof of Theorem \ref{1008@thm1}.\\

\noindent
\emph{Proof of Theorem \ref{1008@thm1}}
For the a priori estimate, by moving all the lower-order terms to the right-hand side of the system, we may assume that  all the lower order coefficients $A^{\alpha\beta}$, $|\alpha|+|\beta|<2m$, are zero.
Then we prove the estimate \eqref{1017@e3a}  using Lemma \ref{1016@lem5} and the standard partition of unity argument.
The  details are omitted.

For the solvability in weighted Sobolev spaces $\mathring{\cH}^m_{p,w}(\Omega_T)$,  we use the idea in \cite[Section 8]{MR3812104} together with Lemma \ref{160612@lem1} below, where the solvability of the system in unweighted Sobolev spaces is proved.
Because the proof is the same as that of Theorem \ref{1016@thm1}, 
we omit the details here.
\qed

\begin{lemma}		\label{160612@lem1}
Let $\Omega$ be a domain in $\bR^d$, $T\in (-\infty,\infty]$, $p_1\in (1,\infty)$, and $\vec f_\alpha\in L_{p_1}( \Omega_T)$, $|\alpha|\le m$.
Then there exist constants 
\begin{align*}
&\gamma_1=\gamma_1(d,m,n,\delta,p_1)\in (0,1/4),\\
&\lambda_1=\lambda_1(d,m,n,\delta,p_1, R_0,K)>0
\end{align*}
such that, under Assumption \ref{0923.ass1} $(\gamma_1)$,
for any $\lambda\ge \lambda_1$, there exists a unique  $\vec u\in \mathring{\cH}_{p_1}^m(\Omega_T)$ satisfying 
$$
\vec u_t+(-1)^m\cL\vec u+\lambda\vec u=\sum_{|\alpha| \le m}D^\alpha \vec f_\alpha \quad \text{in }\, \Omega_T.
$$
Moreover, $\vec u$ satisfies 
$$
\sum_{|\alpha| \le m}\lambda^{1-\frac{|\alpha|}{2m}}\|D^\alpha \vec u\|_{L_{p_1}(\Omega_T)}\le N\sum_{|\alpha| \le m}\lambda^{\frac{|\alpha|}{2m}}\|\vec f_\alpha\|_{L_{p_1}(\Omega_T)},
$$
where $N=N(d,m,n,\delta,p_1)$.
\end{lemma}

\begin{proof}
Thanks to the a priori estimate in Theorem \ref{1008@thm1} with $w\equiv 1$ and the method of continuity, we only need to consider the solvability of the system with simple coefficients, for instance,
$$
\cL \vec u=\sum_{|\alpha|=|\beta|=m}D^\alpha(A^{\alpha\beta}D^\beta \vec u),
$$
where $A^{\alpha\beta}$ are constant.
For this result, we refer to \cite[Theorem 8.2]{MR2835999}, where the authors proved the solvability of the system with more general coefficients (partially BMO coefficients).
Thus the lemma is proved.

It is  worth mentioning that one may prove the solvability of the system by showing that a unique solution in $\mathring{\cH}^m_{2}(\Omega_T)$ is indeed in $\mathring{\cH}^m_{p_1}(\Omega_T)$ with $p_1>2$ and $\vec f_\alpha\in L_2(\Omega_T)\cap L_{p_1}(\Omega_T)$.
An ingredient of this reasoning is to use a reverse H\"older's inequality together with estimates as in Lemma \ref{0929-lem1} with $q=2$ for the solution in $\mathring{\cH}^m_2(\Omega_T)$.
Then by adapting the level set argument (without weights) used in Section \ref{sec6}, one can show that the solution in $\mathring{\cH}^m_2(\Omega_T)$ is in $\mathring{\cH}^m_{p_1}(\Omega_T)$.
For $p_1 \in (1,2)$, we use the usual duality argument.
\end{proof}

We now turn to the proof of Theorem \ref{1016@thm1}.
The proof is based on the weighted $L_p$-estimates obtained in Theorem \ref{1008@thm1} and the following theorem, which is  a refined version of the extrapolation theorem. 
The well-known version of the theorem (see, for instance, \cite{MR2797562}) requires the inequality \eqref{eq0222_01} to hold for all $w \in A_p$.
However, the theorem below allows us to obtain \eqref{eq0222_02} for a given $w \in A_p$ by only checking the inequality \eqref{eq0222_01} for a subset (determined by $p,q,K_0, K_2$) of $A_p$.
This refinement is needed because the weighted $L_p$-estimate \eqref{1017@e3a} holds only for $w$ satisfying $[w]_{A_p}\le K_0$.
See the proof of Theorem \ref{1016@thm1} below.

\begin{theorem}[Extrapolation theorem]		\label{1017@@thm1}
Let  $\cX$ be a space of homogeneous type in $\bR^{d+1}$ or $\bR^d$ with a doubling
constant $K_2$.
Let $p,\, q\in (1,\infty)$, $K_0 \ge 1$, $w\in A_q(\cX)$, and $[w]_{A_q}\le K_0$.
Then there exists a constant $\cK_0=\cK_0(p,q,K_0,K_2)\ge 1$ such that if
\begin{equation}
							\label{eq0222_01}
\|f\|_{L_{p,\tilde{w}}(\cX)}\le N_0\|g\|_{L_{p,\tilde{w}}(\cX)}
\end{equation}
for every $\tilde{w}\in A_p(\cX)$ satisfying $[\tilde{w}]_{A_p}\le \cK_0$, then we have 
\begin{equation}
							\label{eq0222_02}
\|f\|_{L_{q,w}(\cX)}\le 4N_0\|g\|_{L_{q,w}(\cX)}.
\end{equation}
\end{theorem}

\begin{proof}
See \cite[Theorem 2.5]{MR3812104}.
\end{proof}

\noindent
\emph{Proof of Theorem \ref{1016@thm1}}
We first prove the a priori estimate \eqref{1017@e4}.
Assume $\vec u\in C^\infty_0((-\infty,T]\times \Omega)$ satisfies
\[
\vec u_t+(-1)^m \cL\vec u+\lambda\vec u=\sum_{|\alpha| \le m}D^\alpha \vec f_\alpha \quad \text{in }\, \Omega_T.
\]
Let $\tilde{w}_2\in A_p(\cX_2)$ with $[\tilde{w}_2]_{A_p}\le \cK_0$, where  $\cK_0=\cK_0(p,q,K_0,K_2'')\ge 1$ is the constant in Theorem \ref{1017@@thm1}.
Notice from Lemma \ref{160629@lem1} that  $\cX_1\times \cX_2$ is a space of homogeneous type in $\bR^{d+1}$ with a doubling constant $K_2=K_2(K_2'K_2'')$, and
\[
\tilde{w}=w_1\tilde{w}_2\in A_p(\cX_1\times \cX_2) \quad \text{with}\quad [\tilde{w}]_{A_p}\le \tilde{K}_0,
\]
where $\tilde{K}_0=\tilde{K}_0(p,q,K_0,K_2',K_2'')$.
By applying Theorem \ref{1008@thm1}, there exist constants 
\begin{align*}
&\gamma=\gamma(d,m,n,\delta,p,q,K_0,K_2',K_2'')\in (0,1/6),\\
&\lambda_0=\lambda_0(d,m,n,\delta,p,q, K_0,K_2',K_2'',R_0,K)>0
\end{align*}
such that, under Assumption \ref{0923.ass1} $(\gamma)$, we have 
\begin{equation}		\label{1017@@eq1}
\sum_{|\alpha| \le m}\lambda^{1-\frac{|\alpha|}{2m}}\|D^\alpha \vec u\|_{L_{p,\tilde{w}}(\Omega_T)}\le N_1\sum_{|\alpha| \le m}\lambda^{\frac{|\alpha|}{2m}}\|\vec f_\alpha\|_{L_{p,\tilde{w}}(\Omega_T)},
\end{equation}
where $\lambda\ge \lambda_0$ and $N_1=N_1(d,m,n,\delta,p,q,K_0,K_2',K_2'')$.
Set  
\begin{align*}
&\psi(t,x'')=\sum_{|\alpha|\le m}\lambda^{1-\frac{|\alpha|}{2m}}\| I_{\Omega_T}D^\alpha \vec u(t,\cdot,x'')\|_{L_{p,w_1}(\cX_1)},\\
&\phi(t,x'')=\sum_{|\alpha| \le m}\lambda^{\frac{|\alpha|}{2m}}\| I_{\Omega_T}\vec f_\alpha(t,\cdot,x'')\|_{L_{p,w_1}(\cX_1)}.
\end{align*}
It follows from \eqref{1017@@eq1} that  
\begin{equation*}
\|\psi\|_{L_{p,\tilde{w}_2}(\cX_2)}\le N_2\|\phi\|_{L_{p,\tilde{w}_2}(\cX_2)},
\end{equation*}
where $N_2=N_2(d,m,n,\delta,p,q,d_1,d_2,K_0,K_2',K_2'')$. 
Since the above inequality is satisfied for any $\tilde{w}_2\in A_p(\cX_2)$ with $[\tilde{w}_2]_{A_p}\le \cK_0$, by Theorem \ref{1017@@thm1}, we have 
\[
\|\psi\|_{L_{q,w}(\cX_2)}\le 4N_2\|\phi\|_{L_{q,w}(\cX_2)},
\]
which gives the estimate \eqref{1017@e4}.

For the solvability in $\mathring{\cH}^m_{p,q,w}(\Omega_T)$ of the system \eqref{160621@eq1}, we use the argument in \cite[Section 8]{MR3812104}. 
Owing to Lemma \ref{160621@lem1} and the a priori estimate, we may assume that $\vec f_\alpha\in L_{p,q,w}(\Omega_T)\cap L_\infty(\Omega_T)$ with bounded supports.
By Lemma \ref{1016@lem1} and the doubling properties of $\cX_i$ and $w_i$, there exist $\mu_1$ and $\mu_2$ depending only on $p$, $q$, $K_0$, $K_2'$, and $K_2''$ such that 
\begin{equation}		\label{190214@B1}
\mu_1,\mu_2>1, \quad \frac{p\mu_1}{\mu_1-1}=\frac{q\mu_2}{\mu_2-1}=:p_1,
\end{equation}
and 
for any $(x',x'')\in \cX_1\times \cX_2$ and $r>0$, 
\begin{equation}		\label{190213@A1}
\begin{aligned}
\dashint_{\cB_{2r}^{\cX_1}(x')} w_1^{\mu_1}\,dy' &\le N\dashint_{\cB_r^{\cX_1}(x')}w_1^{\mu_1}\,dy',\\
\dashint_{\cB_{2r}^{\cX_2}(x'')} w_2^{\mu_2}\,dy'' &\le N\dashint_{\cB_r^{\cX_2}(x'')}w_2^{\mu_2}\,dy'',
\end{aligned}
\end{equation}
where $N=N(p,q,K_0,K_2',K_2'')>0$.
Set 
$\bar{\gamma}=\min(\gamma,\gamma_1)$ and $\bar{\lambda}_0=\max(\lambda_0,\lambda_1)$,
where $\gamma_1$ and $\lambda_1$ are constants in Lemma \ref{160612@lem1}.
Then by Lemma \ref{160612@lem1}, under Assumption \ref{0923.ass1} $(\bar{\gamma})$, for any $\lambda\ge \bar{\lambda}_0$, there exists a unique $\vec u\in \mathring{\cH}^m_{p_1}(\Omega_T)$ satisfying  \eqref{160621@eq1}.
Moreover, by H\"older's inequality and \eqref{190214@B1}, $\vec u$ is locally in ${\cH}^m_{p,q,w}(\Omega_T)$ with the estimate
\begin{equation}		\label{190214@eq1}
\begin{aligned}
\|D^\alpha \vec u\|_{L_{p,q,w}(Q\cap \Omega_T)}\le \|w_1\|_{L_{\mu_1}(Q')}^{1/p}\|w_2\|_{L_{\mu_2}(Q'')}^{1/q}\|D^\alpha \vec u\|_{L_{p_1}(Q\cap \Omega_T)}
\end{aligned}
\end{equation}
for all compact set $Q\subset Q'\times Q''\subset \cX_1\times \cX_2$ and $|\alpha|\le m$.

To complete the proof, it suffices to show that $\vec u\in \mathring{\cH}^m_{p,q,w}(\Omega_T)$.
Assume that $\vec f_\alpha$ are supported in $\cB_R\cap \Omega_T$ for some $R\ge1$.
For $k\in \{0,1,2,\ldots\}$, let $\eta_k$ be a smooth function on $\bR^{d+1}$ satisfying
$$
0\le \eta_k\le1, \quad \eta_k \equiv 0 \text{ on } \cB_{2^kR}, \quad \eta_k\equiv 1 \text{ on }\, \bR^{d+1}\setminus \cB_{2^{k+1}R},
$$
$$
|(\eta_k)_t|\le N2^{-2mk}, \quad |D^i \eta_k|\le N2^{-ik}, \quad i\in \{0,1,\ldots,m\}.
$$
Then $\eta_k \vec u\in \mathring{\cH}^m_{p_1}(\Omega_T)$ satisfies 
\begin{align*}
&(\eta_k \vec u)_t+(-1)^m \cL (\eta_k \vec u)+\lambda \eta_k \vec u\\
&=(\eta_k)_t\vec u+\sum_{|\alpha|\le m, |\beta|\le m} \sum_{1\le |\beta'|\le |\beta|} c_{\beta,\beta'}D^\alpha (A^{\alpha\beta} D^{\beta'}\eta_kD^{\beta-\beta'}\vec u) \quad \text{in }\, \Omega_T,
\end{align*}
where $c_{\beta,\beta'}$ are appropriate constants.
By applying the a priori estimate in Lemma \ref{160612@lem1}, we have 
\begin{align*}
&\sum_{|\alpha|\le m}\lambda^{1-\frac{|\alpha|}{2m}}\|D^\alpha (\eta_k \vec u)\|_{L_{p_1}(\Omega_T)}\\
&\le N \|(\eta_k)_t \vec u\|_{L_{p_1}(\Omega_T)}+N\sum_{|\alpha|\le m, |\beta|\le m} \sum_{1\le |\beta'|\le |\beta|}\lambda^{\frac{|\alpha|}{2m}}\|D^{\beta' }\eta_k D^{\beta-\beta'}\vec u\|_{L_{p_1}(\Omega_T)}\\
&\le N2^{-k} \sum_{|\alpha|\le m}\lambda^{1-\frac{|\alpha|}{2m}}\|D^\alpha  \vec u\|_{L_{p_1}((\cB_{2^{k+1}R}\setminus \cB_{2^{k}R})\cap \Omega_T)},
\end{align*}
where $N=N(d,m,n,\delta,p_1)$ and  we used the fact that $\lambda\ge 1$ in the last inequality.
Thus, by induction, we obtain that, for $k\ge 1$, 
$$
\begin{aligned}
&\sum_{|\alpha|\le m}\lambda^{1-\frac{|\alpha|}{2m}}\|D^\alpha  \vec u\|_{L_{p_1}((\cB_{2^{k+1}R}\setminus \cB_{2^{k}R})\cap \Omega_T)}\\
&\le N2^{-\frac{k(k-1)}{2}} \sum_{|\alpha|\le m}\lambda^{1-\frac{|\alpha|}{2m}}\|D^\alpha  \vec u\|_{L_{p_1}(\cB_{2R}\cap \Omega_T)}=:2^{-\frac{k(k-1)}{2}}N_1.
\end{aligned}
$$
From this together with \eqref{190213@A1} and \eqref{190214@eq1}, it follows that 
$$
\begin{aligned}
&\sum_{|\alpha|\le m}\lambda^{1-\frac{|\alpha|}{2m}}\|D^\alpha \vec u\|_{L_{p,q,w}((\cB_{2^{k+1}R}\setminus \cB_{2^{k}R})\cap \Omega_T)}\\
&\le \|w_1\|_{L_{\mu_1}(\cB_{2^{k+1}R}^{\cX_1})}^{1/p}\|w_2\|_{L_{\mu_2}(\cB^{\cX_2}_{2^{k+1}R})}^{1/q}\sum_{|\alpha|\le m} \lambda^{1-\frac{|\alpha|}{2m}}\|D^\alpha \vec u\|_{L_{p_1}((\cB_{2^{k+1}R}\setminus \cB_{2^{k}R})\cap \Omega_T)}\\
&\le N_1 N_0^{k}2^{-\frac{k(k-1)}{2}}\|w_1\|_{L_{\mu_1}(\cB_R^{\cX_1})}^{1/p}\|w_1\|_{L_{\mu_2}(\cB_R^{\cX_2})}^{1/q},
\end{aligned}
$$
where $N_0=N_0(p,q,K_0,K_2',K_2'')$. 
This implies that $\vec u\in \mathring{\cH}^m_{p,q,w}(\Omega_T)$.
The theorem is proved.
\qed

\section{Appendix}		\label{sec8}

\begin{lemma}[{\cite[Sec.10.2]{MR0519341}}]			\label{0922.lem2}
Let $r\in (0,\infty)$, $1< q\le p <\infty$, and $k=0,1,\ldots,2m-1$.
Assume that 
\[
\frac{1}{q}-\frac{1}{p}\le \frac{2m-k}{d+2m}.
\]
If $u\in W^{1,2m}_q(Q_r)$, then we have $D^ku\in L_p(Q_r)$ and 
$$
\|D^ku\|_{L_p(Q_r)}\le N\|u\|_{W^{1,2m}_q(Q_r)},
$$
where $N=N(d,m,k,p,q,r)$.
The statement remains true, provided that $Q_r$ is replaced by $Q_r^+$.
\end{lemma}

\begin{lemma}[{\cite[Sec.18.12]{MR0521808}}]		\label{0922.lem4}
Let $r\in (0,\infty)$, $\mu\in (0,1)$, $1<q<\infty$, and $k=0,1,\ldots,2m-1$.
Assume that
\[
q\ge\frac{d+2m}{2m-k-\mu}.
\]
If $u\in W^{1,2m}_q(Q_r)$, then we have $D^ku\in C^\mu(Q_r)$ and 
\begin{equation*}
\|D^k u\|_{C^\mu(Q_r)}\le N\|u\|_{W^{1,2m}_q(Q_r)},
\end{equation*}
where $N=N(d,m,k,q,r,\mu)$.
The statement remains true, provided that $Q_r$ is replaced by $Q_r^+$.
\end{lemma}

\begin{lemma}		\label{160629@lem1}
Let $\cX_1$ and $\cX_2$ are spaces of homogeneous type with doubling constants $K_2'$ and $K_2''$ in $\bR^{d_1}$ and $\bR\times \bR^{d_2}$, $d_1+d_2=d$, respectively.
Then $\cX_1\times \cX_2$ is a space of homogeneous type in $\bR^{d+1}$ with the distance $\rho$ in \eqref{eq0215_01} and  a doubling constant $K_2=K_2(K_2'K_2'')$.
Moreover, if $p\in (1,\infty)$, $K_0',\, K_0''\ge 1$, and 
$$
w(t,x)=w_1(x')w_2(t,x''), \quad x'\in \cX_1, \quad (t,x'')\in \cX_2,
$$
where $w_1\in A_p(\cX_1)$ with $[w_1]_{A_p}\le K_0'$ and $w_2\in A_p(\cX_2)$ with $[w_2]_{A_p}\le K_0''$, then 
$w\in A_p(\cX_1\times \cX_2)$ with $[w]_{A_p}\le K_0=K_0(p,K_0',K_0'',K_2',K_2'')$.
\end{lemma}

\begin{proof}
Denote $\cX=\cX_1\times \cX_2$.
Without loss of generality, we assume that $0\in \cX$.
Since $B_{2r}^\cX\subset B^{\cX_1}_{2r}\times B^{\cX_2}_{2r}$ and $B^{\cX_1}_{r/2}\times B^{\cX_2}_{r/2}\subset B_r^\cX$, by using the doubling property, we have 
$$
|B^\cX_{2r}|\le |B^{\cX_1}_{2r}||B^{\cX_2}_{2r}|\le (K_2'K_2'')^2|B^{\cX_1}_{r/2}||B^{\cX_2}_{r/2}|\le (K'_2K_2'')^2|B_r^\cX|.
$$
Moreover, we obtain
$$
\bigg(\dashint_{B_r^\cX}w\,dx\,dt\bigg)\bigg(\dashint_{B_r^\cX}w^{-\frac{1}{p-1}}\,dx\,dt\bigg)^{p-1}\le \bigg(\frac{|B_r^{\cX_1}||B_{r}^{\cX_2}|}{|B_r^\cX|}\bigg)^pK_0'K_0''
$$
$$
\le \bigg(\frac{|B_r^{\cX_1}||B_{r}^{\cX_2}|}{|B_{r/2}^{\cX_1}||B^{\cX_2}_{r/2}|}\bigg)^pK_0'K_0''\le (K_2'K_2'')^pK_0'K_0''.
$$
The lemma is proved.
\end{proof}

The following lemma is used to show the estimate \eqref{160804@eq9}.
We recall the notation \eqref{160617@eq2}, and we  point out that when $\Omega$ is a Reifenberg flat with $\gamma\in \big(0,\frac{1}{2}\big)$, the inequality  \eqref{160804@eq9a} is valid with $N_0=N_0(d)$. 

\begin{lemma}
							\label{lem0221_1}	
Let $T\in (-\infty,\infty]$, $\Omega$ be a domain in $\bR^d$,  and $\Omega_T\subseteq\cX$, where  $\cX$ is an open set in $\bR^{d+1}$ and a space of homogeneous type with a doubling constant $K_2$.
Assume that there exist $R_0\in (0,1]$ and $N_0>0$ such that 
\begin{equation}		\label{160804@eq9a}
|\cC_R(X)|\ge N_0R^{d+2m}, \quad \forall X\in \Omega_T, \quad \forall R\in (0,R_0].
\end{equation}
Let $p\in (1,\infty)$, $K_0\ge 1$, $w\in A_p(\cX)$, and $[w]_{A_p}\le K_0$.
Suppose that  $D_0$ and $D_1$ are Borel sets satisfying $D_0\subset D_1\subset \Omega_T$, and that there exists a constant $\varepsilon\in (0,1)$ such that the following hold:
\begin{enumerate}[$(i)$]
\item
$w(D_0\cap \cB_{R}^{\cX}(X))<\varepsilon w(\cB^\cX_{R}(X))$ for $X\in \Omega_T$ and $R\ge R_0$.
\item
For any $X\in \Omega_T$ and for all $R\in (0, R_0]$ with $w(\cB^\cX_R(X)\cap D_0)\ge \varepsilon w(\cB_R^{\cX}(X))$, we have $\cC_R(X)\subset D_1$.
\end{enumerate}
Then we obtain 
\[
w(D_0)\le N\varepsilon w(D_1),
\]
where $N=N(d,p, K_0,K_2,N_0)$.

\end{lemma}

\begin{proof}
We first claim that for almost every $X\in D_0$, there exists  $R_X\in (0, R_0)$ such that 
\begin{equation*}
w(D_0\cap \cB^{\cX}_{R_X}(X))=\varepsilon w(\cB_{R_X}^{\cX}(X))
\end{equation*}
and 
\begin{equation}		\label{160804@eq9b}
w(D_0\cap \cB^{\cX}_{R}(X))< \varepsilon w(\cB_R^{\cX}(X)), \quad \forall R\in (R_X, R_0].
\end{equation}
We define a function $\rho$ on $[0,R_0]$  by 
\[
\rho(r)=\frac{w(D_0\cap \cB^{\cX}_r(X))}{w(\cB^{\cX}_r(X))}=\frac{1}{w(\cB_r^{\cX}(X))}\int_{\cB_r^{\cX}(X)}I_{D_0}w\,dY.
\]
By applying the Lebesgue differentiation theorem and using the fact that 
$$
w(X)>0 \quad \text{almost every }\, X\in \cX, 
$$
we obtain for almost every $X\in D_0$ that 
$$
\rho(0)=\lim_{r\to 0^+}\left(\frac{|\cB_r^{\cX}(X)|}{w(\cB_r^{\cX}(X))}\times \dashint_{\cB_r^{\cX}(X)}I_{D_0}w\,dx\right)=1.
$$
Since $\rho$ is continuous on $[0, R_0]$ and $\rho(R_0)<\varepsilon$, there exists $r_X\in (0,R_0)$ such that $\rho(r_X)=\varepsilon$.
Then we obtain the claim by setting
$$
R_X:=\max\{r_X\in (0,R_0):\rho(r_X)=\varepsilon\}.
$$
Hereafter, we denote 
$$
\Gamma=\big\{B^{\cX}_{R_X}(X):X\in D_0'\big\},
$$
where  $D_0'$ is  the set of all points $X\in D_0$ such that $R_X$ exists.
Then by the Vitali lemma, we have a countable subcollection $G$ in $\Gamma$ such that  
\begin{enumerate}[(a)]
\item
$Q\cap Q'=\emptyset$ for any $Q,\, Q'\in G$ satisfying $Q\neq Q'$.
\item
$D_0\subset \bigcup_{\cB_{R_X}^{\cX}(X)\in G}\cB_{5R_X}^{\cX}(X)$.
\end{enumerate}
Indeed, the subcollection $G$ can be constructed as follows.
We write $\Gamma_1=\Gamma$, and  choose a cylinder $\cB_{R_{X_1}}(X_1)$, denoted by $Q^1$, in  $\Gamma$ such that 
$$
R_{X_1}>\frac{1}{2}\sup_{\cB^{\cX}_{R_X}(X)\in \Gamma_1}R_X,
$$
and split $\Gamma_1=\Gamma_2\cup \Gamma_2^*$, where 
$$
\Gamma_2=\{Q\in \Gamma_1: Q^1\cap Q=\emptyset\}, \quad \Gamma_2^*=\{Q\in \Gamma_1:Q^1\cap Q\neq \emptyset\}.
$$
Assume that $Q^k$ and  $\Gamma_{k+1}$ have been already determined.
If $\Gamma_{k+1}$ is empty, then the process ends.
If not,  we choose a cylinder $Q^{k+1}=\cB^{\cX}_{R_{X_{k+1}}}(X_{k+1})$ in $\Gamma_{k+1}$ such that 
$$
R_{X_{k+1}}>\frac{1}{2}\sup_{\cB_{R_X}^{\cX}(X)\in \Gamma_{k+1}}R_X,
$$
and split $\Gamma_{k+1}=\Gamma_{k+2}\cup \Gamma_{k+2}^*$, where
$$
\Gamma_{k+2}=\{Q\in \Gamma_{k+1}: Q^{k+1}\cap Q=\emptyset\}, \quad \Gamma_{k+2}^*=\{Q\in \Gamma_{k+1}:Q^{k+1}\cap Q\neq \emptyset\}.
$$
We define $G=\{Q^k\}_{k\in J}$, where $J\subseteq \bN$.
Obviously, $G$ satisfies $(a)$.
To see $(b)$, we note that 
$$
Q\subset \cB^{\cX}_{5R_{X_{k}}}(X_k), \quad \forall Q\in \Gamma_{k+1}^*, \quad \forall k\in J,
$$
and 
$$
\Gamma_1=\bigcup_{k\in J}\Gamma_{k+1}^*.
$$
Therefore, we have 
$$
D_0\subset \bigcup_{Q\in \Gamma_1} Q\subset \bigcup_{\cB^{\cX}_{R_X}(X)\in G}\cB_{5R_X}^{\cX}(X),
$$
which implies that $G$ satisfies $(b)$.

Now we are ready prove the lemma.
From the assumption $(i)$, \eqref{160804@eq9b}, and the doubling property of $\cX$, it follows that  
$$
w(D_0\cap \cB_{5R_X}^{\cX}(X))<\varepsilon w(\cB_{5R_X}^{\cX}(X))\le \varepsilon (K_2)^3w(\cB^{\cX}_{R_X}(X)), \quad \forall X\in D_0'.
$$
Using this together with $(b)$, we obtain 
\begin{align*}
w(D_0)&=w\Biggl(\bigcup_{\cB^\cX_{R_X}(X)\in G}D_0\cap \cB_{5R_X}^{\cX}(X)\Biggr)\\
&\le \sum_{\cB_{R_X}^{\cX}(X)\in G}w(D_0\cap \cB_{5R_X}^{\cX}(X))\\
&\le \varepsilon (K_2)^3\sum_{\cB_{R_X}^{\cX}(X)\in G} w(\cB_{R_X}^{\cX}(X)).
\end{align*}
Observe that Lemma \ref{1015@lem3} and \eqref{160804@eq9a} yield
$$
w(\cB_{R_X}^{\cX}(X))\le K_0\bigg(\frac{|\cB_{R_X}^{\cX}(X)|}{|\cC_{R_X}(X)|}\bigg)^pw(\cC_{R_X}(X))\le N w(\cC_{R_X}(X)),
$$
where $N=N(d,K_0,p,N_0)$.
By combining the above two inequalities, and then, using $(a)$ and the assumption $(ii)$, we have 
$$
w(D_0)\le \varepsilon N\sum_{\cB_{R_X}^{\cX}(X)\in G}w(\cC_{R_X}(X))
$$
$$
=\varepsilon Nw\Biggl(\bigcup_{\cB^{\cX}_{R_X}(X)\in G}\cC_{R_X}(X)\Biggr)\le \varepsilon N w(D_1),
$$
where $N=N(d,K_0,K_2,p,N_0)$.
The lemma is proved.
\end{proof}

In the lemma below, we prove that  functions in weighted $L_p$ spaces can be approximated by bounded functions.

\begin{lemma}		\label{160621@lem1}
Let $T\in(-\infty,\infty]$, $\Omega$ be a domain in $\bR^d$, and $\Omega_T\subseteq \cX_1\times \cX_2$, where $\cX_1$ and $\cX_2$ are  spaces of homogeneous type  in $\bR^{d_1}$ and $\bR\times \bR^{d_2}$, $d_1+d_2=d$, respectively.
Let $p,\, q\in (1,\infty)$ and 
$$
w(t,x)=w_1(x')w_2(t,x''), \quad x'\in \cX_1, \quad (t,x'')\in \cX_2, 
$$
where $w_1\in A_p(\cX_1)$ and $w_2\in A_q(\cX_2)$.
Then  for given $f\in L_{p,q,w}(\Omega_T)$, there exists a sequence $\{f_k\}_{k=1}^\infty$ in $L_{p,q,w}(\Omega_T)\cap L_\infty(\Omega_T)$ with bounded supports such that $f_k\to f$ in $L_{p,q,w}(\Omega_T)$ as $k\to \infty$.
\end{lemma}

\begin{proof}
Since $fw_1^{1/p}w_2^{1/q}\in L_{p,q}(\Omega_T)$, there exists a sequence $\{g_k\}_{k=1}^\infty$ in $C^\infty_0(\Omega_T)$ such that 
\begin{equation}		\label{160621@Eq1}
g_k\to fw_1^{1/p}w_2^{1/q} \quad \text{in }\, L_{p,q}(\Omega_T).
\end{equation}
Let us define 
$$
f_k=g_kw_1^{-1/p}w_2^{-1/q}I_{M^1_k\times M^2_k}, 
$$
where 
$$
M^1_k=\left\{x'\in \cX_1:w_1(x')\ge \frac{1}{k}\right\}, \quad M^2_k=\left\{(t,x'')\in \cX_2:w_2(t,x'')\ge \frac{1}{k}\right\}.
$$
Note that $f_k$ are bounded functions on $\Omega_T$ and 
$$
\int_{\cX_2}\bigg(\int_{\cX_1}|f_k-f|^pI_{\Omega_T}w_1\,dx'\bigg)^{q/p}w_2\,dx''\,dt
$$
$$
\le \int_{M_k^2}\bigg(\int_{M_k^1}|f_k-f|^pI_{\Omega_T}w_1\,dx'\bigg)^{q/p}w_2\,dx''\,dt
$$
$$
+ \int_{\cX_2}\bigg(\int_{\cX_1\setminus M_k^1}|f|^pI_{\Omega_T}w_1\,dx'\bigg)^{q/p}w_2\,dx''\,dt
$$
$$
+ \int_{\cX_2\setminus M_k^2}\left(\int_{\cX_1}|f|^pI_{\Omega_T}w_1\,dx'\right)^{q/p}w_2\,dx''\,dt:=I^1_k+I^2_k+I^3_k.
$$
It follows from \eqref{160621@Eq1} that $I_k^1\to 0$ as $k\to \infty$.
Since $|\cX_1\setminus M_k^1|\to 0$ as $k\to \infty$, we have 
$$
|f|^pI_{\Omega_T} w_1I_{\cX_1\setminus M^1_k}\to 0 \quad \text{a.e. in }\, \cX_1.
$$
Using the dominated convergence theorem, we obtain that  
$$
\int_{\cX_1\setminus M_k^1}|f|^pI_{\Omega_T}w_1\,dx'\to 0 \quad \text{a.e. in  }\cX_2,
$$
and thus, by the dominated convergence theorem again, we obtain $I^2_k\to 0$ as $k\to \infty$.
Similarly, we obtain that $I_k^3\to 0$ as $k\to \infty$.
The lemma is proved.
\end{proof}

\bibliographystyle{plain}

\end{document}